\documentclass[a4paper, 11pt,reqno]{amsart}
\usepackage{amsfonts, amsthm, amssymb, amsmath, stackengine, scalerel}
\usepackage{subfig}
\usepackage{mathrsfs,array,tikz-cd}
\usepackage{eucal,fullpage,times,color,enumerate,accents, comment}
\usepackage[all]{xy}
\usepackage{url}
\usepackage{hyperref}
\usepackage{enumitem}
\usepackage[new]{old-arrows}
\usepackage{extpfeil}
\usepackage{verbatim}  %for comment environment
\usetikzlibrary{graphs,decorations.pathmorphing,decorations.markings}

\input xy
\xyoption{all}

\newextarrow{\xbigtoto}{{20}{20}{20}{20}}
   {\bigRelbar\bigRelbar{\bigtwoarrowsleft\rightarrow\rightarrow}}

\newcommand{\calF}{\mathcal{F}}

\newcommand{\an}{\mathrm{an}}

\renewcommand{\inf}{\mathrm{inf}}

\newcommand{\opp}{\mathrm{op}}    % opposite category

\newcommand{\Loc}{\mathrm{Loc}}

\newcommand{\Berk}{\mathrm{Berk}}

\newcommand{\Frm}{\textbf{Frm}}
\newcommand{\argu}{\text{(---)}}

\newcommand{\frap}{\mathfrak{p}}

\newcommand{\qp}{Q_{+}}
\newcommand{\rf}{\mathsf{rad}_\F}
\newcommand{\Alin}{\mathcal{A}_{\mathrm{Lin}}}
\newcommand{\A}{\mathcal{A}}
\newcommand{\M}{\mathcal{M}}
\newcommand{\AffBerk}{\mathbb{A}_{\mathrm{Berk}}}

\newcommand{\Z}{\mathbb{Z}}
\newcommand{\N}{\mathbb{N}}
\newcommand{\Q}{\mathbb{Q}}
\newcommand{\R}{\mathbb{R}}

\newcommand{\F}{\mathcal{F}}
\newcommand{\C}{\mathbb{C}}
\newcommand{\vv}{\mathsf{v}}
\newcommand{\thT}{\mathbb{T}}
\newcommand{\ACVF}{\mathrm{ACVF}}

\newcommand{\Sp}{\mathrm{Sp}}
\newcommand{\lav}{\overleftarrow{av}}

\newcommand{\baseS}{\mathcal{S}}

\newcommand{\cosimp}[3]{\xymatrix@1{#1 \ar@<.4ex>[r] \ar@<-.4ex>[r] & {\ }#2 \ar@<0.8ex>[r] \ar[r] \ar@<-.8ex>[r] & {\ } #3 \ar@<1.2ex>[r] \ar@<.4ex>[r] \ar@<-.4ex>[r] \ar@<-1.2ex>[r] & \cdots }}
\newsavebox{\pullback}
\sbox\pullback{%
	\begin{tikzpicture}%
	\draw (0,0) -- (1ex,0ex);%
	\draw (1ex,0ex) -- (1ex,1ex);%
	\end{tikzpicture}}

\definecolor{quotemark}{gray}{0.7}
\makeatletter
\newlength\origparskip

\newcommand{\fquote}{%
	\@ifnextchar[{\fquote@i}{\fquote@i[]}%]
}

\def\fquote@i[#1]{%
	\@ifnextchar[{\fquote@ii{#1}}{\fquote@ii{#1}[]}%]
}%

\def\fquote@ii#1[#2]{%
	\def\pqm@tempa{#1}%
	\def\pqm@tempb{#2}%
	\noindent
	\list
	{}
	{\setlength{\leftmargin}{0.3\textwidth}%
		\setlength{\rightmargin}{0.1\textwidth}%
		\setlength{\origparskip}{\parskip}}%
	\item[]%
	\begin{picture}(0,0)%
	\put(-15,-8){\makebox(0,0){\scalebox{4}{%
				\textcolor{quotemark}{\textquotedblright}}}}%
	\end{picture}%
	\begingroup
	\itshape
	\ignorespaces}%

\def\endfquote{%
	\endgroup
	\par
	\raggedleft
	\ifx\pqm@tempa\empty
	\else
	{\bfseries --- \pqm@tempa\par}%
	\setlength{\parskip}{\origparskip}%
	\ifx\pqm@tempb\empty
	\else
	(\pqm@tempb)%
	\fi
	\fi
	\par
	\endlist}
\makeatother

\newcommand\norm[1]{\left\lVert#1\right\rVert}

\newcommand{\equalizer}[2]{\xymatrix@1{#1 \ar@<.4ex>[r] \ar@<-0.4ex>[r] & {\ } #2}}

\newcommand{\adjunction}[4]{\xymatrix@1{#1{\ } \ar@<-0.3ex>[r]_{ {\scriptstyle #2}} & {\ } #3 \ar@<-0.3ex>[l]_{ {\scriptstyle #4}}}}
\begin{document}
\bibliographystyle{alpha}
\newtheorem{theorem}{Theorem}[section]
\newtheorem*{theorem*}{Theorem}
\newtheorem*{condition*}{Condition}
\newtheorem*{definition*}{Definition}
\newtheorem*{corollary*}{Corollary}
\newtheorem{proposition}[theorem]{Proposition}
\newtheorem{lemma}[theorem]{Lemma}
\newtheorem{corollary}[theorem]{Corollary}
\newtheorem{claim}[theorem]{Claim}
\newtheorem{conclusion}[theorem]{Conclusion}
\newtheorem{hypothesis}[theorem]{Hypothesis}
\newtheorem{summarytheorem}[theorem]{Summary Theorem}

\theoremstyle{definition}
\newtheorem{definition}[theorem]{Definition}
\newtheorem{question}{Question}
\newtheorem{remark}[theorem]{Remark}
\newtheorem{observation}[theorem]{Observation}
\newtheorem{discussion}[theorem]{Discussion}
\newtheorem{guess}[theorem]{Guess}
\newtheorem{example}[theorem]{Example}
\newtheorem{condition}[theorem]{Condition}
\newtheorem{warning}[theorem]{Warning}
\newtheorem{notation}[theorem]{Notation}
\newtheorem{construction}[theorem]{Construction}
\newtheorem{problem}[theorem]{Problem}
\newtheorem{fact}[theorem]{Fact}
\newtheorem{thesis}[theorem]{Thesis}
\newtheorem{convention}[theorem]{Convention}
\newtheorem{summary}[theorem]{Summary}
\newtheorem{reminder}[theorem]{Reminder}

\newtheorem{naivedefinition}[theorem]{Naive Definition}
\newtheorem*{theorem:BerkovichDISC}{Theorem A}

\title{Logical Berkovich Geometry: A Point-free Perspective}
\author{Ming Ng}
    \thanks{\emph{Thanks:} Research partially supported by EPSRC Grant EP/V028812/1.} 
    
\begin{abstract} 
	Extending our insights from \cite{NVOstrowski}, we apply point-free techniques to sharpen a foundational result in Berkovich geometry. In our language, given the ring $\A:=K\{R^{-1}T\}$ of convergent power series over a suitable non-Archimedean field $K$, the points of its Berkovich Spectrum $\M(\A)$ correspond to $R$-good filters. The surprise is that, unlike the original result by Berkovich, we do not require the field $K$ to be non-trivially valued. Our investigations into non-Archimedean geometry can be understood as being framed by the question: what is the relationship between topology and logic? 

\end{abstract}

\maketitle

\begin{fquote}[E. Hrushovski and F. Loeser \cite{HruLoe}]
	Model theory rarely deals directly with topology; the great exception is the theory of o-minimal structures, where the topology arises naturally from an ordered structure.
\end{fquote}

\begin{fquote}[S. Vickers \cite{Vi3}]
	While geometric logic can be treated as just another logic, it is an unusual one. [...] To put it another way, the geometric mathematics has an intrinsic continuity.
\end{fquote}

It is well-known that any complex algebraic variety\footnote{More precisely, a scheme of (locally) finite type over $\C$.} $X$ can be canonically associated to a complex analytic space $X^{\an}$ via a (functorial) construction known as \emph{complex analytification}. This opens up the study of complex algebraic varieties to powerful tools in complex analysis and differential geometry, prompting the natural question: can we play the same game for algebraic varieties over fields which are not $\C$? For instance, over non-Archimedean fields (e.g. the complex $p$-adics $\C_p$)?

The general thrust of these questions is challenging, but over-simplistic. It is over-simplistic because the naive analytification of algebraic varieties over non-Archimedean fields loses significant information about the original variety, limiting its intended usefulness. Still, it is challenging because it brings into focus the main issue behind this lossy-ness: unlike the complex numbers $\C$, a non-Archimedean field $K$ is totally disconnected. Once understood and made precise, this tells us where to start looking for a robust non-Archimedean analogue of complex analytification.

The key premise of Berkovich geometry \cite{BerkovichMonograph} is that the naive analytification of non-Archimedean varieties is disconnected because we do not have enough points. The solution then, by way of a construction known as \emph{Berkovich analytification}, is to fill in those missing points before developing techniques to study these new analytic spaces.\footnote{There are also other solutions to this disconnectedness problem, e.g. Tate's rigid analytic geometry, which involves defining an appropriate Grothendieck topology that finitises the usual notion of a topology. See e.g. \cite[\S 1.5]{Payne}.} What is interesting to us, however, is how further study of these Berkovich spaces often involve a re-characterisation of the original construction. Consider, for instance, the following characterisations of the Berkovich Affine Line:

\begin{summarytheorem} %\emph{(Equivalent Characterisations of $\AffBerk^1$)} 
	\label{sumthm:BERKOVICH}
	As our setup,
	\begin{itemize}
		\item Fix $K$ to be an algebraically closed field complete with respect to a non-trivial non-Archimedean norm $|\cdot|$;
		\item Denote $\Gamma$ to be the value group of $K$, which we shall assume to be contained in $\R$;
		\item Denote $\AffBerk^1$ to be the Berkovich affine line.
	\end{itemize}	
	\noindent \underline{Then}, $\AffBerk^1$ can be equivalently characterised as:

	\begin{enumerate}[label=(\roman*)]
		\item The set of multiplicative seminorms on $K[T]$ extending $|\cdot|$ on $K$, equipped with the Berkovich topology;
		\item A space whose points are defined by a sequence of nested closed discs $D_{r_1}(k_1)\supseteq D_{r_2}(k_2)\supseteq\dots $ contained in $K$;
		\item The space of types over $K$, concentrating on $\mathbb{A}^1_K$, that are ``almost orthogonal to $\Gamma$'';\footnote{This is a technical definition which the non-logician may wish to treat as a black box --- it will not be needed to understand the main results of this thesis. For the model theorist: $K$ here is taken to be a model of the theory $\ACVF$, which is a 3-sorted theory comprising $\mathrm{VF}$ as the value field sort, $\Gamma$ as the value group sort, and $\kappa$ as the residue field sort. Let $\mathbb{U}$ be the monster model. If $C\subset \mathbb{U}$ and $p$ is a $C$-type, then we say $p$ is \emph{almost orthogonal to $\Gamma$} if for any realisation $a$ of $p$ we have that $\Gamma(C(a))=\Gamma(C)$.}
		\item A profinite $\R$-tree.
	\end{enumerate}
\end{summarytheorem}
\begin{proof}(i) is the definition of $\AffBerk^1$. 
	(ii) can be proved similarly to \cite[Example 1.4.4]{BerkovichMonograph}. (iii) is a special case of what was proved in \cite[pp. 187-188]{HruLoe}. (iv) essentially follows from \cite[Theorem 2.20]{BakerRumeley}. %and the fact that $\mathbb{P}^1_{\Berk}$ is the 1-point compactification of $\AffBerk^1$ at the level of topological spaces.
	%,  which shows that the Berkovich projective line $\mathbb{P}^1_{\Berk}$ is a profinite $\R$-tree, and the fact that $\mathbb{P}^1_{\Berk}=\AffBerk^1\cup\{\infty\}$ %, and the fact that $\mathbb{P}^1_{\Berk}$ is the 1-point compactification of $\AffBerk^1$.
\end{proof}

For the algebraic geometer, the different characterisations of $\AffBerk^1$ in Summary Theorem~\ref{sumthm:BERKOVICH}
reflect the variety of methods that have been used when studying Berkovich spaces. In more detail:

%to Berkovich spaces to obtain important results.
%reflect the variety of tools that have been employed to obtain key results about Berkovich spaces.

\begin{itemize}
	\item The equivalence of items (i) and (ii), a foundational result in Berkovich geometry, sets up the classification of points of $\AffBerk^1$. %, allowing us to, e.g. measure the (spherical) completeness of the base field $K$;
	
	\item The language of ``almost orthogonal types'' reflects the model-theoretic methods pioneered by Hrushovski and Loeser \cite{HruLoe}. This perspective was particularly useful for establishing the topological ``tameness'' of Berkovich spaces under very mild hypotheses (see e.g. Theorem~\ref{thm:BerkLOCALCONTRACT}).
	\begin{comment}
	
	\item The language of ``almost orthogonal types'' reflects the model-theoretic methods pioneered by Hrushovski and Loeser \cite{HruLoe}. This perspective was particularly useful for establishing the topological ``tameness'' of Berkovich spaces under very mild hypotheses. A beautiful example is their result that for any quasi-projective $K$-scheme $V$ and any semi-algebraic subset $U$ of its analytification $V^\an$, there exists a deformation retract of $U$ onto a finite simplicial complex $\Delta\subset U$.\footnote{Previously, it was only known that the generic fibre of a polystable formal $K$-scheme admits a strong deformation retraction to a closed subset homeomorphic to a finite simplicial complex, a result due to Berkovich \cite{BerkovichPolystable}. }

	\begin{itemize}
	\item Loeser: Hrushovski-Loeser results work when trivially valued cos they embed in trivially valued fields...
	\item See Payne's article, check whether I explained this correctly;
	\item Is my comparison between generic fibre and semi-algebraic set justified?
	\end{itemize}
	\end{comment}
	
	%%% See Baker-Rumely Section 1.4, pro-finite R-tree is defined by having a unique arc joining any two points, and this arc has to be geodesic; this emerges from the least upper bound semi-lattice structure. 
	\item Viewing $\AffBerk^1$ as a profinite $\R$-tree emphasises its semilattice structure: given any $x,y\in\AffBerk^1$ there exists a unique least upper bound $x\vee y\in \AffBerk^1$ [with respect to the partial order that $x\leq y$ iff $|f|_x\leq |f|_y$ for $f\in K[T]$]. A key insight of Baker and Rumely \cite{BakerRumeley} was that this structure on $\AffBerk^1$ (in fact, on $\mathbb{P}^1_{\Berk}$) can be used to define a Laplacian operator, laying the foundations for a non-Archimedean analogue of complex potential theory. Their work later found surprising applications in the analysis of preperiodic points of complex dynamical systems \cite{BakerDeMarco}. % \footnote{The fact that non-Archimedean analytic techniques, developed in analogy with the complex case, should find applications in the non-Archimedean setting is reasonable; what is perhaps less expected is that these non-Archimedean techniques should also find applications in the complex setting. \cite{BakerDeMarco} gives an example of this in complex dynamics, but non-Archimedean methods have also been useful when studying complex algebraic varieties. For details, we recommend \cite[\S 5]{Payne}.} 
	%%% See xvi-xvii in Baker-Rumely for discussion on this, and also, Section 1.6.
	%%% Say something about other approaches to NA potential theory.
	
	\begin{comment} They proved that the set of parameters $c\in\C$ such that both $a$ and $b$ are preperiodic for $z^d+c$ is infinite iff $a^d=b^d$.
	\end{comment}
\end{itemize}
\noindent For the topos theorist, however, Summary Theorem~\ref{sumthm:BERKOVICH} is suggestive because it mirrors the different representations of a point-free space: as a universe of (algebraic) models axiomatised by a first-order theory, as a certain space of prime filters, or as a distributive lattice. One may therefore wonder if the listed characterisations of $\AffBerk^1$ reflect a constellation of perspectives on Berkovich spaces that move together in a tightly-connected way. 

It is this intuition that will guide us to the main result of this paper, Theorem~\ref{thm:BerkovichDISC}, where point-free techniques are used to reformulate the equivalence of items (i) and (ii) in Summary Theorem~\ref{sumthm:BERKOVICH}.\footnote{Technically, Theorem~\ref{thm:BerkovichDISC} works with multiplicative seminorms on the ring of convergent power series $K\{R^{-1}T\}$ and not those on $K[T]$, but in fact the result extends to the latter setting by Remark~\ref{rem:AffineLineNOTBerkSpectrum}.} The main surprise is that, unlike Berkovich's original result, the point-free approach works equally well for both the trivially and non-trivially valued fields. This indicates that the \emph{algebraic} hypothesis of $K$ being non-trivially valued is in fact a \emph{point-set} hypothesis, and is not essential to the underlying mathematics. Our result thus gives an interesting proof of concept regarding the clarifying potential of the point-free perspective within non-Archimedean geometry.

\subsection*{How to read this paper} For the reader primarily interested in non-Archimedean geometry \& its topology: no category theory or logic will be needed to understand the proof of the main result. As such, this reader may wish to examine the definition of an upper real (Definition~\ref{def:upperreal}), review  Example~\ref{ex:BerkPOWERseriesRING} to recall why the radii of rigid discs fail to be well-defined when the base field is trivially-valued, before proceeding to Section~\ref{sec:BerkClassThm}. Looming in the background, however, are a series of deep interactions between topology \& logic, which we discuss more fully in Section~\ref{sec:ptfree}. The curious non-logician may wish to make note that filters play a role in these interactions, and treat the rest as a black box. The model theorist, however, may be interested to learn that topos theorists work with an infinitary fragment of positive logic known as \emph{geometric logic}, and what shows up as \textit{types} in the context of model theory sometimes shows up as honest \textit{models} of a geometric theory (e.g. Dedekind reals). Clarifying this connection in the setting of non-Archimedean geometry motivates a very interesting series of test problems, which we discuss in Section~\ref{sec:algforkinroad}.

Finally, for those interested in topos theory: our result in Berkovich geometry can be regarded as an advertisement for point-free topology, in particular, how point-free techniques may resolve a problem by eliminating some of the set-theoretic noise. For those interested in constructive mathematics, we remark that the full strength of Theorem~\ref{thm:BerkovichDISC} relies on LEM. A weaker version of our result holds (Theorem~\ref{thm:Alin=rgoodFILTER}) if one wishes to avoid classical assumptions, but there's some fine print -- see Summary~\ref{sum:classicalVSgeometric}.

\subsection*{Acknowledgements} This paper benefitted from various mathematical exchanges with Matt Baker, Rob Benedetto, Vladimir Berkovich, Artem Chernikov, Eric Finster, Udi Hrushovski, Fran\c{c}ois Loeser, Michael Temkin, Steve Vickers, and Vincent (Jinhe) Ye. Expository aspects of this work also benefitted from presentation at the ``Type Theory, Constructive Mathematics and Geometric Logic'' Conference at CIRM, as well as at the Logic Working Group hosted by the University of Torino -- thank you all.

\begin{comment}
\begin{itemize}
\item Slightly classical
\item List results
\end{itemize}

\footnote{Remark: These different notions of spaces implicit in the Summary Theorem also give rise to different notions of completions. Item (ii) ``completes'' the disconnected space $K$ via sequences of nested closed discs; item (iii) considers the types on $\mathbb{A}^1_K$ which can read as considering the limit points of the underlying model (for details on how types of a model can be viewed as limit points, see \cite{MalliarisICM}); finally, by construction, item (iv) regards $\AffBerk^1$ as an inverse limit of finite $\R$-trees.} 
%%% Incompleteness: non-definable types, Type IV points.
% 
\end{comment}
% \tableofcontents

\section{Preliminaries}\label{sec:prelim}

\subsection{The Point-free Perspective}\label{sec:ptfree} As a first approximation, point-free topology can be understood as the study of localic spaces.

\begin{definition}\label{def:frame} \hfill
	\begin{enumerate}[label=(\roman*)]
		\item 	A \emph{frame} is a complete lattice $A$ possessing all small joins $\bigvee$ and all finite meets $\land$, such that the following distributivity law holds
		\[a\land \bigvee S = \bigvee \{a\land b \,|\, b\in S\}\]
		where $a\in A, S\subseteq A$.
		\item 	A \emph{frame homomorphism} is a function between frames that preserves arbitrary joins and finite meets.
	\end{enumerate}	
	Frames and frame homomorphisms form the category $\Frm$. We define the category $\Loc:=\Frm^{\opp}$. We refer to the objects in $\Loc$ as \emph{localic spaces}.\footnote{Elsewhere, this is often referred to as a \emph{locale}; our terminology was chosen for suggestiveness.} %%% Notationally, distinguish $\Omega$.
\end{definition}

The main difference between the localic perspective and point-set topology is one of priority: what are the basic units for defining a space? In the case of a point-set space, one starts with a set of elements, before defining the topology on it in the usual way. On the other hand, the localic perspective treats the opens as the basic units for defining a space. Once the topological structure has been defined, we can then recover its points:

\begin{definition}[Points]\label{def:points} Let $\Omega X$ be a frame. Denote $\Omega:=\Omega\textbf{1}$ to be the powerset of the singleton.
	\begin{enumerate}[label=(\roman*)]
		\item A \emph{global point} is a frame homomorphism $\Omega X \to \Omega$.
		\item A \emph{generalised point} is a frame homomorphism $\Omega X\to\Omega Y$, where $\Omega Y$ is any frame.
	\end{enumerate}
\end{definition}

Definition~\ref{def:frame} states that localic spaces and frames are the same thing, but conceptually we shall prefer to view \emph{frames} as corresponding to the opens of a space and the localic space as corresponding to the universe of all its (generalised) points. Hereafter, whenever we have a localic space $X$, we shall denote its corresponding frame of opens as $\Omega X$. 

\begin{remark}\label{rem:filtersPTS} Classically, $\Omega$ of course corresponds to the two element Boolean algebra $\{0,1\}$. As such, $x\colon \Omega X \to \Omega$ can be regarded as a way of sorting out which opens the point $x$ belongs to and which ones it doesn't. However, the assertion $\Omega\textbf{1}\cong \{0,1\}$ relies on the Law of Excluded Middle (LEM); in particular, if we opt to work constructively, then the isomorphism no longer holds.
% $\Omega$ is often regarded as the frame of truth values. %% , telling us that something is either true or false
\end{remark}

% More details can be found in, e.g. \cite{NgThesis} or \cite{VickersPtfreePtwise}.

\subsubsection{Geometric Logic}\label{sec:geomlogic} The view from topos theory emphasises the logical nature of this point-free perspective. We start with the notion of a geometric theory:

\begin{definition}[Geometric Theories] Let $\Sigma$ be a (many-sorted) first-order signature (or vocabulary).
	\begin{itemize}
	\item Let $\vec{x}$ be a finite vector of variables, each with a given sort. A \emph{geometric formula} in context $\vec{x}$ is a formula built up using symbols from $\Sigma$ via the following logical connectives: $=$, $\top$ (true), $\land$ (\textbf{finite} conjunction), $\bigvee$ (\textbf{arbitrary} disjunction), $\exists$.
 \item  A \emph{geometric theory over $\Sigma$} is a set of axioms of the form
		\[\forall\vec{x}. (\phi\to \psi),\]
		where $\phi$ and $\psi$ are geometric formulae.
	\end{itemize}    
\end{definition}

\begin{remark}\label{rem:geomsyntax} The geometric syntax already reveals key differences with classical first-order logic, namely:
	\begin{itemize}
		\item The absence of negation $\lnot$
		\item Allowing arbitrary (possibly infinite) disjunction
		\item Disallowing nested implications and universal quantification in the axioms.
	\end{itemize}
	This not only affects what we are able to express in this logical language, but also how we understand ``truth''. In classical logic, one may appeal to the LEM and assert that $p\vee \lnot p$ for any proposition $p$. In geometric logic, we are unable to even express this statement due to the lack of negation. 
\end{remark}

\begin{convention} Hereafter, the unqualified term ``theory'' shall always mean a geometric theory. If we wish to refer to theories from classical first-order logic, we shall signpost this explicitly.
\end{convention}

%Examples of geometric theories include 

Of particular interest to us is a special class of geometric theories known as \emph{essentially propositional theories}. Before explaining the qualifier \emph{essentially}, we start with the basic definition:
 %We give an informal definition (full details can be found in, e.g. \cite[\S2]{NgThesis} or \cite{VickersPtfreePtwise})

\begin{definition}\label{def:prop} A (geometric) theory $\thT$ is called a \textit{propositional theory} if its signature $\Sigma$ has no sorts [so there can be no variables or terms, nor existential quantification].  In particular, its axioms are constructed only from constant symbols in $\Sigma$, $\top$ (true), finite $\land$ and arbitrary $\bigvee$.  
\end{definition}

The connection between geometric logic and locales now becomes apparent: the propositional formulae of $\thT$ correspond to the opens, the finite $\land$ to the finite intersections of opens and the arbitrary $\bigvee$ to their arbitrary unions. Further, given any frame $A$, one can define a theory $\thT_A$ such that its models correspond to the points of the corresponding localic space.

\begin{definition}\label{def:frameProp} Let $A$ be a frame. Define $\thT_A$ as follows: 
	\begin{itemize}
		\item[] \textbf{Signature} $\Sigma$: A propositional symbol $P_a$ for each $a\in  A$
		\item[] \textbf{Axioms}:
\begin{enumerate}
		\item $P_a\to P_b$, for $a\leq b$ in $A$
	\item $P_a\land P_b\to P_{a\land b}$, for $a,b\in A$
	\item $\top\to P_1$
	\item $P_{\bigvee S}\to \bigvee_{a\in S}P_a$, for $S\subseteq A$
\end{enumerate}
	\end{itemize}
\end{definition}
These axioms ensure that meets and joins of the frame correspond to conjunctions and disjunctions in the logic, in particular that frame $A$ corresponds to the Lindenbaum Algebra of $\thT_A$. Given two propositional theories $\thT$ and $\thT'$, define a \emph{$\thT$-model in $\thT'$} to be a Lindenbaum algebra homomorphism between their corresponding Lindenbaum algebras $\Omega_\thT\to \Omega_{\thT'}$. A straightforward check then verifies that a $\thT_A$-model in $\thT_B$ corresponds to frame homomorphism $A\to B$, for any frames $A,B$ (for details, see \cite[\S 2]{Vi4}).

Developing Remark~\ref{rem:filtersPTS}, we can regard the models of $\thT_A$ as \emph{completely prime filters}. Recall: a \emph{filter} is a collection of subsets satisfying certain formal properties also satisfied by the collection of open neighbourhoods of any point $x\in X$ in a point-set topological space. The hypothesis of \emph{completely prime} enforces a kind of coherence within the filter with respect to $\bigvee$. More precisely:

\begin{definition}[Filters]\label{def:filters} Let $S$ be an infinite set. %%% Does it have to be infinite?
	\begin{enumerate}[label=(\roman*)]
		\item A \emph{filter on $S$} is a collection of subsets $\F\subseteq \mathcal{P}(S)$ such that:
		\begin{enumerate}[label=(\alph*)]
			\item $A\subseteq B\subseteq S$ and $A\in\F$ implies $B\in\F$;
			\item $A,B\in\F$ implies $A\cap B\in \F$; and
			\item $S\in \F$
		\end{enumerate}
		\item A filter $\F$ is \emph{prime} if for every finite set index set $I$:
		\begin{itemize}
			\item[] $\bigcup_{i}A_i\in\F$ implies there exists some $j\in I$ such that $A_j\in\F$.
		\end{itemize}
		A filter $\F$ is \emph{completely prime} if the same holds true for \emph{any} index set $I$ (including when $I$ is infinite).
		\item A filter $\F$ is called an \emph{ultrafilter} if it has an opinion on all subsets of $S$:
		\begin{itemize}
			\item[] If $A\subset S$, then either $A$ or its complement $S\setminus A$ belongs to $\F$ (but not both).
		\end{itemize}
	\end{enumerate}
\end{definition}

The translation of Definition~\ref{def:filters} to the setting of locales is straightforward (although there's a subtlety when it comes to ultrafilters\footnote{Why? Clearly the logical counterpart to taking the complement of a subset should be taking negation, but negation is not part of the geometric syntax. Recalling that the Lindenbaum Algebra of a propositional theory corresponds to a frame, we can also understand this topologically: the opens of a topological space are generally not clopen. See also Discussion~\ref{dis:typesMODELSfilters}.
}). Since every frame $A$ gives rise to a propositional theory $\thT_A$, and the Lindenbaum Algebra of every propositional theory $\thT$ gives rise to a frame $\Omega_\thT$, this justifies the view that the universe of all $\thT$-models for any propositional $\thT$ can be regarded as a localic space, and that these $\thT$-models correspond to completely prime filters.\footnote{More precisely, they correspond to the \emph{standard models} of $\thT_A$ (equivalently, the global points of $A$) that correspond to completely prime filters. A standard model of $\thT_A$ tells us which propositional symbols $P_a$ are assigned truth value $\top$, and hence describes a subset $F\subseteq A$ satisfying the axioms of $\thT_A$. Re-examining $\thT_A$ in light of Definition~\ref{def:filters}, the first three axioms tell us that $F$ is a filter, the fourth tells us that $F$ is completely prime.}

How does this picture extend to predicate theories? This is where the topos theory comes in. Here's the basic gist. Given any (geometric) theory $\thT$, one uses categorical logic to define what a $\thT$-model means in this setting. The topos theory then says: the universe of all $\thT$-models can be regarded as a kind of generalised space -- let us say a \emph{point-free space} -- whose points correspond to the $\thT$-models. We shall treat this framework as a black box in this paper (for details on the underlying topos theory, see e.g. \cite[Ch. 2]{NgThesis}, or \cite{Vi4, VickersPtfreePtwise}). Let us instead record two guiding remarks which the reader may find illuminating. First, compare this point-free perspective with the model-theoretic perspective. For the model theorist, the theory of groups $\thT_{\mathrm{grp}}$ gives rise to a class of structures satisfying its axioms: the elementary class of groups.  By contrast, in topos theory, the theory of groups gives rise to a (pseudo-)functor of models, analogous to the ``functor of points'' approach in algebraic geometry (see \cite[B4.2]{J1}). It is in this sense that the universe of all $\thT_{\mathrm{grp}}$-models may be regarded as a generalised space.

 Second, topos theory also tells us when these generalised spaces of models are equivalent.\footnote{More precisely, two theories $\thT,\thT'$ define equivalent spaces of models iff their classifying toposes $\baseS[\thT]\simeq \baseS[\thT']$ are equivalent. For details, see \cite[Prop. 2.1.28]{NgThesis}.} We can therefore define a more general class of theories, including the original propositional theories, whose space of models are equivalent to those of a propositional theory. Call such theories \emph{essentially propositional}, and let us also refer to their corresponding space of models as \emph{localic spaces}.\footnote{Why consider essentially propositional theories? In practice, the presence of sorts in a theory's signature often allows for a much nicer axiomatisation of the models, whereas the absence of sorts indicates that the theory is less logically complex and so potentially easier to work with. By working with essentially propositional theories (as opposed to just propositional theories), we enjoy the best of both worlds.}
 An important class of essentially propositional theories are those whose sorts are all free algebra constructions --- e.g. the natural numbers $\N$, the integers $\Z$, the rationals $\Q$. A key example would be the localic reals, which we turn to next.

\subsubsection{Localic Reals} This paper uses two different types of reals: the usual Dedekind reals, and the upper reals. Both reals are built up from the rationals but in different ways; this results in different topologies and therefore different subtleties in their analysis. 

\begin{definition}\label{def:ThDed} Define the propositional theory of Dedekind reals $\thT_\R$ as follows:
\begin{itemize}
	\item[] \textbf{Signature} $\Sigma$: Propositional symbols $P_{q,r}$, where $q,r\in\Q$
	\item[] \textbf{Axioms}:
	\begin{enumerate}
		\item $P_{q,r}\land P_{q',r'}\leftrightarrow \bigvee \{P_{s,t} \,|\, \max(q,q')<s<t<\min(r,r')\}$
		\item $\top\rightarrow \bigvee \{P_{q-\epsilon,q+\epsilon}\,|\, q\in \Q\}$
 if $0<\epsilon\in\Q$
\end{enumerate} 
\end{itemize}
\end{definition}

\begin{discussion}[The points of $\thT_\R$]\label{dis:R} Notice again the difference with point-set topology. In point-set topology, we start with the underlying set of reals before equipping it with the usual topology; in the point-free setting, we start with a collection of symbols $\{P_{qr}\}_{q,r\in\Q}$, before applying axioms to ensure that they formally behave like the rational-ended open intervals in $\R$. The usual Dedekinds can then be obtained as completely prime filters of the Lindenbaum Algebra of $\thT_\R$. This illustrates what was mentioned before: the basic units for defining the space $\R$ are the opens, no longer its underlying set of points. 
\end{discussion}

\begin{definition}[Upper \& Lower Reals]\label{def:upperreal} Denote $x$ to be a subset of the rationals. For suggestiveness, write $x<q$ to mean $q\in x$. An \emph{upper real} is a subset $x\subseteq \Q$ that satisfies:
	\begin{enumerate}
		\item \emph{Upward-closure.} If $x<q$ and $q<q'$, then $x<q'$.
		\item \emph{Roundedness.} If $x<q$, then there exists $q'<q$ such that $x<q'$.
	\end{enumerate}
	for all $q,q'\in \Q$.	Analogously, a \emph{lower real} is a subset $x'\subseteq \Q$ that is downward-closed and rounded.
\end{definition}

In fact, we can extend Definition~\ref{def:upperreal} to give a more explicit description of the Dedekinds. %%% Actually, maybe this should go after defining upper & lower; Dedekind reals are $(L,U)$ s.t. rounded and downward/upward closed. Plus separated + located.

\begin{definition}\label{def:Ded2} A Dedekind real is a pair $(L,R)$ of subsets of $\Q$ whereby
\begin{itemize}
	\item $R$ is an \emph{inhabited}\footnote{That is, there exists $r\in\Q$ such that $R<r$.} upper real, and $L$ is an \emph{inhabited} lower real.
	\item $(L,R)$ are \emph{separated}: $(L,R)$ are disjoint subsets of $\Q$.
	\item $(L,R)$ are \emph{located}: for any rationals $q,r$, either $q<L$ or $R<r$.
\end{itemize}	
\end{definition}

\begin{remark}\label{rem:essPROP} We have suppressed the syntactic details, but it is clear that Definitions~\ref{def:upperreal} and \ref{def:Ded2} define predicate theories featuring $\Q$ as a sort in their signatures. Nonetheless, since $\Q$ is a free algebra construction, the theories are actually essentially propositional. One can also deduce this explicitly for Definition~\ref{def:Ded2} by checking that its space of models is equivalent to those of $\thT_\R$ -- see \cite[\S 4.7]{Vi4}.
\end{remark}

\begin{convention} We denote the space of Dedekinds as $\R$. By Remark~\ref{rem:essPROP}, the points of $\R$ may be regarded as models of $\thT_\R$ or the theory described by Definition~\ref{def:Ded2}. For suggestiveness, whenever we want to work with a generic Dedekind $x$, we write $x\in \R$.  
\end{convention}

We conclude with some basic facts about the upper reals. Notice Definition~\ref{def:upperreal} allows the upper real $x=\emptyset$ or $x=\Q$ (corresponding to $\infty$ and $-\infty$ respectively). Excluding these cases, an upper real can be thought of as approximating a number from above whereas a Dedekind real approximates the number from both above and below. In particular: 

\begin{fact}\label{fact:UpperDed} Denote $\overleftarrow{[\infty,\infty)}$ to be the space of upper reals excluding $\infty$.\footnote{In other words, the space of inhabited upper reals -- see the first bullet point of Definition~\ref{def:Ded2}.}
\begin{enumerate}[label=(\roman*)]
	\item  There exists a natural map
	\begin{align*}
	R\colon \mathbb{R} &\longrightarrow \overleftarrow{[-\infty,\infty)}
	\end{align*}
	where given a Dedekind real $x=(L_x,R_x)$, $R$ sends $x\mapsto R_x$. One can define an analogous map from the Dedekinds to the lower reals by sending $x\mapsto L_x$. We remark that the maps $x\mapsto R_x$ and $x\mapsto L_x$ are monic, but they do not define embeddings: $\R$ is most certainly not a subspace $\overleftarrow{[-\infty,\infty)}$ due to the difference in topology. For details, see \cite[Corollary 1.27]{NV}.
	\item Given an indexed set of reals $\{x_i\}_{i\in I}$, we define 
	\[\underset{i\in I}{\inf}\,x_i := \bigcup_{i\in I} \{ q\in \Q \,| \, x_i<q  \}\]
	\[\sup_{i\in I}x_i := \bigcup_{i\in I} \{ q\in \Q \,| \, q<x_i  \}\]
	Classically, bounded infima/suprema of Dedekinds are regarded as defining Dedekinds; constructively, however, this is no longer the case. Examining the above definitions, notice that a set-indexed $\inf$ of reals (Dedekind or upper) defines an upper real whereas a set-indexed $\sup$ of reals (Dedekind or lower) defines a lower real. For details on why bounded upper/lower reals cannot be (constructively) considered as Dedekinds, see Discussion~\ref{dis:classASSBerk}.
\end{enumerate}

\end{fact}

We enforce the following conventions when describing subspaces of the upper reals; the reader unfamiliar with Scott-continuity may wish to treat them as a black box (for details, see e.g. \cite[\S1]{NVOstrowski} or ~\cite{Vi4}).

\begin{convention}[Subspaces of Upper Reals]\label{conv:specorder}\hfill 
	\begin{enumerate}[label=(\roman*)]
		\item The space of upper reals comes equipped with a Scott topology: $x\sqsubseteq y$ iff $x\geq y$. Notice that $\sqsubseteq$ is opposite to the usual numerical order. To reflect this, we use arrows on top of the spaces to show the direction of the refinement under their respective specialisation orders -- e.g. $\overleftarrow{[0,\infty)}$, $\overleftarrow{[-\infty,\infty)}$ etc.
		\item Notice that the previous one-sided intervals were closed at the arrowhead --- e.g. $0$ was included in $\overleftarrow{[0,\infty)}$ and $-\infty$ in $\overleftarrow{[-\infty,\infty)}$. In fact, this is canonical --- \textit{all} upper real intervals must be closed at the arrowhead. Why? Answer: the upper reals possess the Scott topology, and so all subspaces of the uppers must be closed under arbitrary directed joins with respect to $\sqsubseteq$.
	\end{enumerate}
	
\end{convention}

\subsubsection{Interlude: Classical vs. Point-free Perspective}\label{subsec:interlude} The claim that Dedekind reals can be characterised as models of a first-order theory will be provocative to the model theorist. In classical model theory, Dedekind reals typically arise not as models but as \textit{types} %\footnote{A warning: the model-theoretic notion of types is different from the notion of types used by type-theorists.}
over the model $M=(\mathbb{Q},<)$, i.e. the rationals considered as a dense linear order. 
We contextualise this via the language of filters:

%We contextualise this difference in perspective via the language of filters:
% maybe rephrase last sentence? Want to say smth like that there is something more general going on, and we provide the context for this differencce in logical perspectives via the language of filters. But this is too woolly.

\begin{discussion}[Types vs. Models as Filters]\label{dis:typesMODELSfilters}   Informally, a (complete) type $p$ over a model $M$ corresponds to an ultrafilter of the Boolean Algebra of definable subsets of $M$, which we denote as $\mathcal{B}_M$. Analogously, Section~\ref{sec:geomlogic} tells us that models of a propositional theory $\thT$ correspond to the completely prime filters of its Lindenbaum Algebra $\Omega_\thT$. The appearance of filters in both contexts is suggestive, but there is a subtlety. Since the geometric syntax does not have negation (cf. Remark~\ref{rem:geomsyntax}), the Lindenbaum Algebra $\Omega_{\thT}$ is generally not Boolean (and so its completely prime filters are generally not ultrafilters either). However, when $\Omega_{\thT}$ is in fact Boolean, then the prime filters of $\Omega_{\thT}$ turn out to be precisely its ultrafilters and so the two notions coincide.\footnote{Some care, however, needs to be taken regarding the distinction between prime vs. completely prime filters.} 
	%%% Completely prime filters have to be principal ultrafilters..
	%\footnote{Why? Recall: all prime filters of a Boolean lattice are also ultrafilters, and all completely prime filters are a fortiori prime filters. % In fact, the completely prime filters of a Boolean algebra are precisely the completely regular ultrafilters. Johnstone pp. 134 [check]},
	
	We can also phrase this in the language of points (in the sense of Definition~\ref{def:points}). A complete type over $M$ may be characterised as a Boolean homomorphism $\mathcal{B}_M\to\{0,1\}$ to the two-element Boolean algebra, whereas a global point of a localic space $[\thT]$ corresponds to a frame homomorphism $\Omega_{\thT}\rightarrow\Omega$, where $\Omega$ is the frame of truth values. Classically, these notions more or less coincide since LEM implies $\Omega\cong\{0,1\}$, but constructively this is no longer the case (cf. Remark~\ref{rem:filtersPTS}).

	%\footnote{There's an interesting question of what corresponds model-theoretically to the topos theorist's notion of a generalised (localic) point, i.e. a frame homomorphism $\Omega_{[\thT]}\rightarrow \Omega_{[\thT']}$ for propositional theories $\thT,\thT'$. On this, see \cite{BooleanTypes} for details on the so-called ``Boolean types''. See also \cite[\S 2.3]{HruLoe} for a discussion on the definable type as Boolean retractions.} 
\end{discussion}

\begin{discussion}\label{dis:LogicalComplexity} Viewed logically, the example of the Dedekind reals brings into focus a challenging connection between the model theorist's type spaces and the topos theorist's point-free spaces. Parsing their similarities and differences reflects the contrasting legacies of Shelah vs. Grothendieck on the development of modern logic.

	On the one hand, the model theorist and the topos theorist have developed very different understandings of logical complexity. Whereas model theorists typically tie the logical complexity of the theory to combinatorial questions (e.g. number of types, number of non-isomorphic models etc.), the topos theorist typically ties the logical complexity of a theory to its syntactic expressiveness (e.g. coherent, geometric, propositional vs. predicate etc. --- see \cite[Remark D1.4.14]{J2}). In the case of Dedekind reals, the model theorist regards them as evidence that the theory $\mathrm{Th}(\Q,<)$ is complex [more precisely, that $\mathrm{Th}(\Q,<)$ is unstable], whereas the topos theorist regards the theory of Dedekind reals as being particularly well-behaved [more precisely, it is an essentially propositional theory]. %%% The fact that we work with categorical logic introduces a nuance to the question: what does it mean to be equally expressive? categorical equivalence introduces a nuance: we also consider categories of models. An equivalence of categories of models is not the same as saying there is an isomorphism of elementary classes. See the discussion on propositional vs. essentially propositional theories; see also Boney-Vasey discussion. 
	
	 On the other hand, there appears to be some convergence in attitudes regarding the desired geometry of these logical spaces (compare e.g. the Galois-theoretic ideas in Joyal-Tierney's monograph \cite{JoyalTierney} with Hrushovski's recent preprints \cite{HrushovskiDefinable,HrushovskiLascar}). 
	Further development of these structural connections appear important, and will be the subject of future work. 
	%% how to extract model-independent information from these
	
	%% Combine test problems from model theory with the structure theorems from topos theory.
	%%% Hrushovski & inert spaces. No automorphisms. Morleyisation + Booleanisation; evidence from Espindola's work indicates that this brings into focus the saturated models of the theory; interesting to see what this means in relation to, e.g. Shelah's universality
	%%% Also, why not implement Berkovich ideas to fill in the holes of the Stone space of types? But maybe that's the wrong way to think about it. The main issue: Stone space is disconnected cos we use Boolean formulas as our opens, which are clopen. What is the analogue of Gelfand-Mazur in this case?
\end{discussion}

\begin{warning}\label{warning:DedekindsNOinfty} One should not jump to conclusions about the space of Dedekinds being equivalent to the Stone space of 1-types $\mathrm{S}_1(\Q,<)$. Unlike the latter, the space of Dedekinds does \textbf{not} contain infinities or infinitesimals --- its models are just the real numbers belonging to the interval $(-\infty,\infty)$. We can also understand the difference topologically: $\mathrm{S}_1(\Q,<)$ is a totally disconnected compact space whereas the space of Dedekind reals $\R$ is connected and non-compact.
\end{warning}

Aside from Warning~\ref{warning:DedekindsNOinfty}, there is a more fundamental reason to be cautious about overextending the informal picture of ``Models of a geometric theory $\approx$ Types arising from model theory''. As already mentioned, completely prime filters are not ultrafilters in general. An illuminating example would be the upper reals from before. 

\begin{discussion}\label{dis:POVonUPPERS} %\emph{(Perspectives on the Upper Real)} 
	Recall that an upper real is a real number that only records the rationals strictly larger than itself and nothing else. This results in some striking differences with the Dedekinds. Certainly the standard upper real does not correspond to a complete type over $(\mathbb{Q},<)$ (unless, of course, the upper real happens to be $\infty$ or $-\infty$). Further, since an upper real is blind to the rationals less than itself, we also cannot say when one upper real is \textit{strictly} smaller/greater than another. Moreover, recall that an upper reals is defined as a well-behaved subset of $\mathbb{Q}$, which are collectively ordered by subset inclusion. As such, the upper reals combine to form a non-Hausdorff space whose points `contain' each other.%\footnote{Why is an upper real not a complete type? Note that the upper real $2$ records all instances of rationals $q$ such that $2<q$. But this set of formulae will still be consistent if we add another formula, say $-4<-$.}
\end{discussion}

\begin{discussion}\label{dis:whyUPPER} 	This paper is interested in analysing the multiplicative seminorms on certain rings, natural in the context of Berkovich geometry. So why might the upper reals be a natural choice to work with as opposed to the usual Dedekinds? The reason will become clearer in due course. For now, just for orientation, let us note that Summary Theorem~\ref{sumthm:BERKOVICH} tells us that multiplicative seminorms on $K[T]$ correspond to a sequence of nested discs $D_{r_1}(k_1)\supseteq D_{r_2}(k_2) \dots $, suggesting that these multiplicative seminorms are being approximated ``from above''. Topologically, this matches more closely with the upper reals than the Dedekinds (cf. our comments before Fact~\ref{fact:UpperDed}.) 
%	 In addition, \cite[Discussion 5.9]{NVOstrowski} observes that Dedekind-valued absolute values on $\Q$ are necessarily positive definite whereas positive-definiteness is not well-defined for upper-valued absolute values on $\Z$. This gives more evidence that if we want to consider \emph{all} multiplicative seminorms, it is more natural (at least from the point-free perspective) to work with upper reals.
\end{discussion}

\subsection{The Berkovich Perspective} Let us review in broad strokes the motivation behind Berkovich geometry. In complex algebraic geometry, complex varieties can be regarded as complex manifolds and thus can be studied using powerful tools from complex analysis and differential geometry (see e.g. \cite[Appendix B]{Hartshorne}). The main barrier to playing the same game for non-Archimedean varieties is that the base field $K$ is totally disconnected. Berkovich's solution to this problem of disconnectedness is to fill in the ``missing points'', which we now discuss.

\subsubsection{\protect\dots on algebraic varieties} Let $(K,|\cdot|)$ be a non-Archimedean field, and $X$ be an affine variety over $K$, i.e. $X$ is the zero locus in $K^n$ of a finite set of polynomials $f_1,\dots,f_m\in K[x_1,\dots,x_n]$. Recall that the coordinate ring of $X$ is defined as 
\begin{equation}
K[X]:=K[T_1,\dots,T_n]/(f_1,\dots,f_m)
\end{equation}
One easily checks that every point $k\in X(K)$ gives rise to a multiplicative seminorm on $K[X]$
\begin{align}
|\cdot|_k\colon K[X] &\longrightarrow \R_{\geq 0}\\
f & \longmapsto |f(k)|, \nonumber 
\end{align}
otherwise known as the \emph{evaluation seminorm at $k$}. Extending this insight, one can then define the analytification of $X$ in terms of seminorms on its coordinate ring.

\begin{definition}[Analytification of Affine Varieties]\label{def:BerkAFFINE}  Fix a valued field $(K,|\cdot|)$, and let $X$ be an affine variety over $K$ with associated coordinate ring $K[X]$.	
	\begin{enumerate}[label=(\roman*)]
		\item Given a multiplicative seminorm on $K[X]$, which we denote 
		\begin{align}
		|\cdot|_x \colon K[X] \longrightarrow \R_{\geq 0}
		\end{align}
		we say that $|\cdot|_x$ \emph{extends the given norm on $K$} if $|k|_x=|k|$ for all $k\in K$. Notice that this includes the evaluative seminorms.
		\item The \emph{analytification of $X$}, denoted $X^{\an}$, is defined as the following point-set space:
		\begin{itemize}
			\item \textit{Underlying set of $X^{\an}$} = the set of all multiplicative seminorms on $K[X]$ extending the original norm $|\cdot|$ on the base field $K$; 
			\item \textit{Topology on $X^{\an}$} = the weakest topology such that all maps of the form
			\begin{align}\label{eq:BerkTopCTS}
			\psi_f\colon X^{\an} &\longrightarrow \R_{\geq 0}\\
			|\cdot|_x&\longmapsto |f|_x.  \nonumber
			\end{align}  
			are continuous, for any $f\in K[X]$, which we shall call the \textit{Berkovich topology}.\footnote{This is also sometimes called the \textit{Gelfand topology}.} For clarity, we emphasise that $\psi_f$ is a mapping on \textit{all} multiplicative seminorms on $K[X]$ extending $|\cdot|$ --- not just the evaluative seminorms from before. 
		\end{itemize}
	\end{enumerate}
\end{definition}

\begin{example} Given a non-Archimedean field $(K,|\cdot|)$, the underlying set of the Berkovich Affine line $\AffBerk^1$ is the set of multiplicative seminorms 
	\begin{equation}
	|\cdot|_x\colon K[T] \longrightarrow \R_{\geq 0}
	\end{equation}
	extending the norm on $K$, as already seen in Summary Theorem~\ref{sumthm:BERKOVICH}.
\end{example}

\begin{remark} The expert reader may have noticed Definition~\ref{def:BerkAFFINE} actually only defines the underlying topological space of the Berkovich analytification. To be correct, let us mention that the full definition of Berkovich analytification also includes a structure sheaf of analytic functions on $X^{\an}$, yielding a locally ringed space. For details, see \cite[Ch. 2-3]{BerkovichMonograph}.
\end{remark}

\begin{discussion}\label{dis:BerkovichTopology} 
	Regarding the topological aspects of Definition~\ref{def:BerkAFFINE}:
	\begin{enumerate}[label=(\roman*)]
		\item Despite its point-set formulation, the Berkovich analytification $X^{\an}$ is suggestive from the point-free perspective since the points of these spaces correspond to models of theories. In particular, let us mention \cite{NVOstrowski}, which highlights a subtle interaction between topology and algebra. Namely, if we define absolute values on $\Z$ using upper reals, then the Scott topology prevents us from defining positive definiteness. Conversely, if we consider absolute values on a field (e.g. $\Q$), then the absolute values must be valued in the Dedekinds -- this in turn can be shown to imply positive definiteness. This extends Discussion~\ref{dis:whyUPPER} on why it may be natural to work with the upper reals.
		\item The Berkovich topology can be more explicitly characterised as the weakest topology such that for all $f\in K[X]$ and for all $\alpha\in \R$, the sets
		\begin{align}\label{eq:BerkTOPbasics}
		U(f,\alpha):=\{ |\cdot|\in X^{\an} \,\big|\, |f|<\alpha \}\\
		V(f,\alpha):=\{ |\cdot|\in X^{\an} \,\big|\, |f|>\alpha \} \nonumber
		\end{align}
		are open in $X^{\an}$. Notice then that Berkovich topology crucially depends on the fact that $|\cdot|$ is valued in the Dedekinds as opposed to say, the upper reals.\footnote{Why? Note that $V(f,\alpha)$ from Equation~\eqref{eq:BerkTOPbasics} would no longer be well-defined since an absolute value $|\cdot|\colon K[X]\rightarrow \overleftarrow{[0,\infty)}$ valued in the upper reals is unable to sense which values are smaller than $|f|$, only those larger than it. See Discussion~\ref{dis:POVonUPPERS}.}
		%%  This follows from \cite[\S 1.1]{BakerRumeley}
		\item Why should we regard the space of multiplicative seminorms as filling in the ``missing points'' of the base field? Notice that Definition~\ref{def:BerkAFFINE} is defined for any valued field $K$, not necessarily non-Archimedean. By the Gelfand-Mazur Theorem, every multiplicative seminorm on $\C[T]$ corresponds to an evaluative seminorm, and so one deduces $\AffBerk^1\cong\C$ when $K=\C$. However, when $K$ is non-Archimedean, then there exist more multiplicative seminorms than just the evaluative seminorms.  %%% Say more about this in the tree? Points at the bottom correspond to evaluative seminorms.
	\end{enumerate}
\end{discussion}

In fact, Definition~\ref{def:BerkAFFINE} can be extended to the more general case of $K$-schemes of (locally) finite type (for details, see \cite[Ch. 2 - 3]{BerkovichMonograph}). One important appeal of the Berkovich analytification is that it constructs well-behaved spaces that are sensitive to the topological character of the original variety. 

\begin{summarytheorem}[{{\cite[\S 3.4-3.5]{BerkovichMonograph}}}]  Let $X$ be a $K$-scheme of finite type. Then $X^{\an}$ is locally compact and locally path-connected. Furthermore, we also have the following GAGA-type results:%\footnote{We remark that these GAGA-type results hold for both when $K$ has a trivial norm and $K$ has non-trivial non-Archimedean norm --- see Theorems 3.4.8 and of 3.5.3 of \cite{BerkovichMonograph}.}
	\begin{enumerate}[label=(\roman*)]
		\item $X$ is a connected scheme iff $X^{\an}$ is connected;
		\item $X$ is a separated scheme iff $X^{\an}$ is Hausdorff;
		\item $X$ is a proper scheme iff $X^{\an}$ is compact.
	\end{enumerate}	
\end{summarytheorem}
\begin{comment} \begin{proof} The fact that fact that any Berkovich $K$-analytic space is locally path-connected follows from \cite[Cor. 2.2.8]{BerkovichMonograph} whereas items (i) - (iii) are \cite[Theorem 3.4.8]{BerkovichMonograph}. Not sure why $X^{\an}$ is locally compact. 
\end{proof}
\end{comment}

%\noindent 
As a beautiful example of the interaction between (classical) logic and Berkovich geometry, let us also mention the following result by Hrushovski and Loeser.

\begin{theorem}[{{\cite[Theorem 14.4.1]{HruLoe}}}]\label{thm:BerkLOCALCONTRACT}  Let $X$ be a $K$-scheme of finite type. Then, its Berkovich analytification $X^{\an}$ is locally contractible. %K nontrivially valued?
\end{theorem}

\noindent Prior to Theorem~\ref{thm:BerkLOCALCONTRACT}, local contractibility was only known in the case of smooth Berkovich analytic spaces \cite{BerkovichPolystable}; by contrast, the model-theoretic techniques developed by Hrushovski and Loeser \cite{HruLoe} were sufficiently general to handle both the singular and non-singular cases. 

%\footnote{Technically, not analytification of Banach rings... maybe use the term spectra}}
\subsubsection{\protect\dots on Banach rings} In algebraic geometry, the basic building blocks are affine schemes, which are spectra of commutative rings. In Berkovich geometry, the analogous building blocks are the so-called Berkovich spectra of commutative Banach rings.
% (\cite[Ch. 1]{BerkovichMonograph}).  
%% analytification of varieties can be traced back to...?
%%% Say something about how K[T] is not a Banach ring. 

\begin{definition}[The Berkovich Spectrum]\label{def:BerkBANACHSpec}  Let $(\A,||\cdot||)$ be a commutative Banach ring\footnote{Recall: a Banach ring $(\A,||\cdot||)$ is a normed ring that is complete with respect to $||\cdot||$.} with identity.
	\begin{enumerate}[label=(\roman*)]
		\item A \emph{bounded} multiplicative seminorm on $\A$ is a multiplicative seminorm
		\begin{equation}
		|\cdot|_x\colon \A\longrightarrow \mathbb{R}_{\geq 0}
		\end{equation}
		that satisfies the inequality $|f|_x\leq ||f||$ for all $f\in\A$.
		\item  The \emph{Berkovich Spectrum $\M(\A)$} is the set of all bounded multiplicative seminorms on $\A$, equipped with the weakest topology such that the map
		\begin{align}
		\psi_f\colon \M(\A)&\longrightarrow \R_{\geq 0}\\
		|\cdot|_x &\longmapsto |f|_x \nonumber 
		\end{align}
		is continuous for all $f\in\A$.
	\end{enumerate}
\end{definition}

\begin{convention}[On the Berkovich Spectrum] \hfill  
	\begin{enumerate}[label=(\roman*)]
		\item In this section, all seminorms should be assumed to be multiplicative unless otherwise stated.
		\item 	The given norm on the Banach ring $\A$ will always be denoted as $||\cdot||$. By contrast, to emphasise that $\M(\A)$ is a topological space, its points will be represented as $|\cdot|_x$, or even $x\in \M(\A)$ when the context is clear. % For clarity, we again emphasise that $|\cdot|_x$ should not be assumed to be an evaluative seminorm on $\A$, unless explicitly stated.
	\end{enumerate}
	
\end{convention}

We illustrate this construction with examples. A few orienting remarks are in order.  First, notice there is nothing specifically non-Archimedean about Definition~\ref{def:BerkBANACHSpec} --- in fact, as we shall see in Example~\ref{ex:BerkINTEGERS}, the Berkovich spectrum of $\Z$ yields a space that naturally includes both Archimedean and non-Archimedean components. Another notable feature of Berkovich geometry is that the basic setup accommodates both the trivially and non-trivially valued fields --- see Example~\ref{ex:BerkPOWERseriesRING}.  Interestingly, this flexibility appears to be abandoned/lost in many modern approaches to the subject (see e.g. \cite{BakerRumeley,BenedettoBook}). Finally, the generality of Definition~\ref{def:BerkBANACHSpec} also allows us to define the Berkovich spectrum of an important class of Banach rings known as $K$-affinoid algebras --- its role in Berkovich geometry will be discussed in Example~\ref{ex:TATEalgebras}.

\begin{example}\label{ex:BerkINTEGERS} The ring of integers $(\Z,|\cdot|_{\infty})$ equipped with the usual Euclidean norm is a Banach ring. The characterisation of its Berkovich spectrum $\M(\Z)$ usually proceeds by a series of case-splittings.
	\begin{itemize}
		\item \textit{Stage 1: Identify the seminorms that fail positive definiteness.}  Any point $x\in M(\Z)$ corresponds to a seminorm $|\cdot|_x$ on $\Z$, which induces a prime ideal $\frap_{x}=\{n\,|\,|n|_x=0\}\subseteq \Z$. In the case where $\frap_x=p\Z$, deduce that $|\cdot|_x$ induces the unique trivial seminorm on $\mathbb{F}_p$, whereby
		\begin{equation}\label{eq:trivialnormFp}
		|\cdot|_x=|n|_{p,\infty}:=\begin{cases}
		0 \qquad \text{if} \, \,p|n\\
		1 \qquad \text{if otherwise}.
		\end{cases}
		\end{equation} % \footnote{Why? Notice that $|\cdot|_x$ induces a multiplicative seminorm on $\Z/\frap_x =\mathbb{F}_p$ essentially by hypothesis. One then easily checks that the only multiplicative norm on $\mathbb{F}_p$ is the one which coincides with Equation~\eqref{eq:trivialnormFp}.}
	Otherwise, note that the induced prime ideal $\frap_x=(0)$ must be the zero ideal.
		\item \textit{Stage 2: Classify the remaining seminorms.} Suppose $\frap_x=(0)$. By Ostrowski's Theorem, deduce that $|\cdot|_x$ must be one of the following:
		\begin{itemize}[label=]
			\item \textit{Case 2a:} $|\cdot|_x=|\cdot|_0$ is the trival norm on $\Z$.
			\item  \textit{Case 2b:} $|\cdot|_x=|\cdot|^{\alpha}_{\infty}$ for some $\alpha\in (0,1]$, where $|\cdot|_{\infty}$ is the usual Euclidean norm.
			\item \textit{Case 2c:} $|\cdot|_x=|\cdot|_p^{\alpha}$ for some $\alpha\in(0,\infty)$, where $|\cdot|_p$ is the standard $p$-adic norm.
		\end{itemize}
	\end{itemize}
	Assembling this data together, one obtains the following picture:%\footnote{This is almost identical to our picture of the topos of upper-valued absolute values on $\Z$ in \cite{NVOstrowski} \textit{except} for the fact that the intervals here are valued in Dedekinds as opposed to upper reals.		}:
	
	\begin{figure}[ht!]
		\centering
		\includegraphics[width=0.7\linewidth]{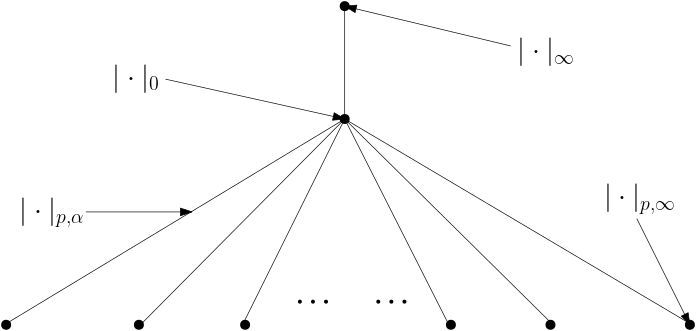}
		\caption{$\M(\Z)$}
		\label{fig:BerkovichSpecIntegers}
	\end{figure}
\end{example} 

%%% This is the completion of K[T] with respect to another norm. See Jonsson.
\begin{example}\label{ex:BerkPOWERseriesRING} Fix an algebraically-closed non-Archimedean field $(K,|\cdot|)$. We define the Banach ring $(\A,||\cdot||)$ whereby:
	\begin{itemize}
		\item $\A$ is the ring of power series converging in radius $R$
		\begin{equation}
		\mathcal{A}=K\{R^{-1}T\}:=\left\{f=\displaystyle\sum_{i=0}^{\infty}c_i T^i\,\Bigg|\, c_i\in K, \displaystyle\lim_{i\rightarrow\infty}|c_i|R^i =0 \right\}
		\end{equation}
		\item $||\cdot||$ is the so-called \emph{Gauss norm} 
		\begin{equation}
		||f||:=\sup_i|c_i|R^i,\quad \text{where $f\in\A$}.
		\end{equation}
	\end{itemize}
	The description of $\M(\A)$ differs depending on whether $K$ is trivially or non-trivially valued.
	\begin{itemize}
		\item \textit{Case 1: $K$ is trivially valued.} In which case,
		\begin{equation}
		K\{R^{-1}T\}=\begin{cases}
		K[[T]] \qquad \,\,\text{if}\, R<1 \\
		K[T] \qquad\quad \text{if}\, R\geq 1
		\end{cases}
		\end{equation}
		where $K[[T]]$ is the formal power series ring and $K[T]$ is the polynomial ring.\footnote{\label{footnote:POWERSERIEStrivialnorm} Why? Notice if $R<1$, then $\displaystyle\lim_{i\rightarrow\infty}|c_i|R^i$ is always zero since $|c_i|=0$ or 1; if $R\geq 1$ instead, then the sequence $\{c_i\}$ must eventually be 0.} When $R<1$, one checks that the the map $|\cdot|_x\mapsto |T|_x$ yields a homeomorphism $\M(\A)\cong [0,R]$; when $R\geq 1$, a more involved argument shows $\M(\A)$ has the structure of $\M(\Z)$ (see Figure~\ref{fig:BerkovichSpecIntegers}).
		\item \textit{Case 2: $K$ is non-trivially valued.} In which case, the characterisation of $\AffBerk^{1}$ in Summary Theorem~\ref{sumthm:BERKOVICH} extends here as well: all points of $x\in\M(\A)$ can be realised as
		\begin{equation}\label{eq:LIMofDISCS}
		|\cdot|_x=\lim_{n\to\infty} |\cdot|_{D_{r_i}(k_i)}
		\end{equation}
		for some nested descending sequence of discs  
		\begin{equation}
		D_{r_1}(k_1)\supseteq D_{r_2}(k_2)\supseteq\dots 
		\end{equation}
		where $|\cdot|_{D_{r}(k)}$ is a multiplicative seminorm canonically associated to the closed disc
		\begin{equation}\label{eq:PtSetcloseddisc}
		D_{r}(k):=\{b\in K \,\big| \, |b-k|\leq r\}.
		\end{equation}
Discs of this form shall be referred to as a \emph{rigid disc}. This description of $\M(\A)$ results in a complicated tree with infinite branching points:	
		
		\begin{figure}[ht!]
			\centering
			\includegraphics[width=0.8\linewidth]{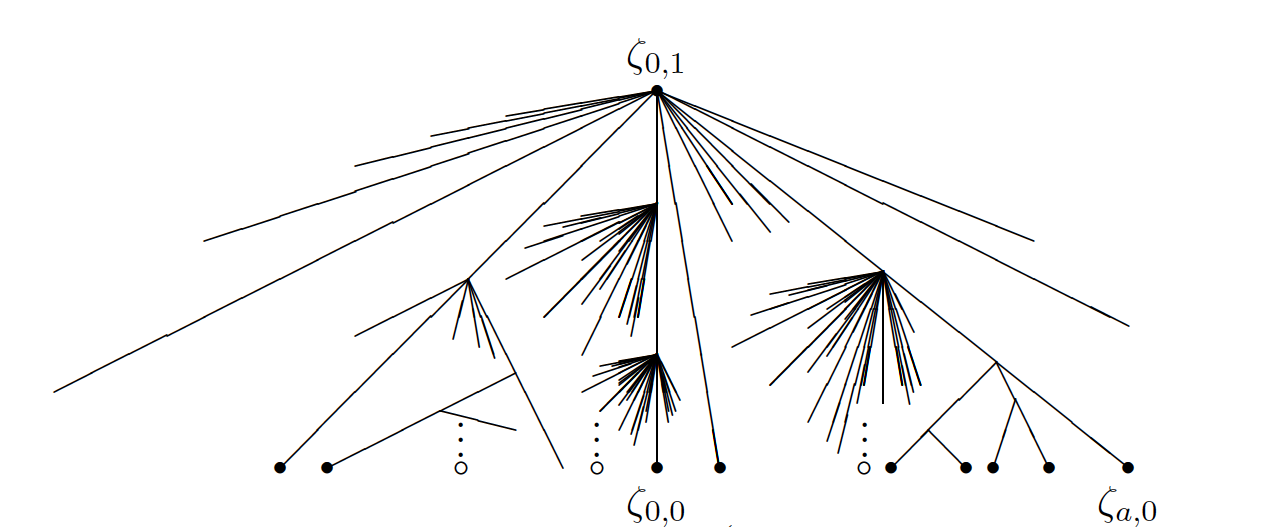}
			\caption{ $\M(K\{R^{-1}T\})$ when $R=1$, adapted from \cite{BakerRumeley,SilvermanDynamical}}
			
		\end{figure}
		
		The reader may wonder: where did we use the hypothesis that $K$ was non-trivially valued? 	Notice that the rigid discs in Equation~\eqref{eq:PtSetcloseddisc} are defined as subsets of $K$.  As such, in order for their radii to be well-defined (i.e. if $D_{r}(k)=D_{r'}(k')$ then $r=r'$) the base field $K$ is forced to be non-trivially valued.\footnote{\label{footnote:RADIIdisc} Why? If $K$ is equipped with a trivial norm, then by definition $|k|=1$ for all $k\neq 0$ in $K$. In which case, $D_{r}(k)=\{k\}$ for any $k\in K$ whenever $r<1$.} In fact, since $|\cdot|_{D_{r}(k)}$ is defined by sending
			\begin{align}
		|f|_{D_{r}(k)}:=\sup_{z\in D_{r}(k)} |f(z)|,
		\end{align}
		 these multiplicative seminorms begin to collapse into one another when $K$ is trivially-valued, causing the original approximation argument to break down.
	\end{itemize}
\end{example}

\begin{remark}\label{rem:AffineLineNOTBerkSpectrum} The Berkovich Affine Line $\AffBerk^1$ is defined as the space of multiplicative seminorms on $K[T]$. Is $\AffBerk^1$ therefore just another example of a Berkovich spectrum? The answer, perhaps surprisingly, is generally no. Recall: Berkovich spectra are defined for \emph{Banach Rings}. When $K$ is non-trivially valued, one can check that $K[T]$ fails to be complete with respect to the Gauss norm $||\sum a_iT_i||=\max_{i}|a_i|$. However, two important caveats: 
	\begin{enumerate}[label=(\alph*)]
		\item One can also check that $K\{R^{-1}T\}$ is in fact a Banach ring (with respect to the appropriate Guass norm), and that $\AffBerk^1$ can be represented as an infinite union of Berkovich spectra
		\[\AffBerk^1\cong \bigcup_{R>0}\M(K\{R^{-1}T\}).\]
		\item When $K$ is trivially valued, then $K[T]$ turns out to be complete with respect to $||\cdot||$ and thus defines a Banach ring; in which case, the two constructions do coincide. This gives another way of reading the difference between the trivially vs. non-trivially valued fields in the Berkovich setting.
	\end{enumerate}
	For details, see e.g. \cite[Example 1.4.4]{BerkovichMonograph} or \cite[Ch. 1-2]{BakerRumeley}. 
\end{remark}

\begin{example}\label{ex:TATEalgebras}  Extending Example~\ref{ex:BerkPOWERseriesRING}, given $R_1,\dots,R_n>0$, define the ring
	\begin{equation}
	K\{R_1^{-1}T_1,\dots, R_n^{-1}T_n\}:=\left\{f=\displaystyle\sum_{\vv\in\N^{n}}a_{\vv} T^\vv\,\Bigg|\, a_\vv\in K, \displaystyle\lim_{|\vv|\rightarrow\infty}|a_\vv|R^\vv =0 \right\},
	\end{equation}
	where $|\vv|=\vv_1+\dots+\vv_n$ and $R^{\vv}=R_1^{\vv_1}\dots R_n^{\vv_n}$. To turn 
	$	K\{R_1^{-1}T_1,\dots, R_n^{-1}T_n\}$ into a Banach ring, we equip it with the Gauss norm $||\cdot||$ where	
	\begin{equation}
	||f||:= \sup_{\vv} |a_{\vv}|R^\vv.
	\end{equation}
	When $R_i=1$ for all $i$, then $K\{R_1^{-1}T_1,\dots, R_n^{-1}T_n\}$ is called the \emph{Tate Algebra}.
	
	These Banach rings play an important role in Berkovich geometry because they allow us to define \emph{$K$-affinoid algebras}, which are the quotients of these power series rings. More precisely, a $K$-affinoid algebra $\A$  is a commutative Banach ring for which there exists an admissible epimorphism 
	\begin{equation}
	\begin{tikzcd}
	K\{R_1^{-1}T_1,\dots, R_n^{-1}T_n\} \ar[r, two heads] & \A.
	\end{tikzcd}
	\end{equation}
	This should be understood as being the analogues of quotients of polynomial rings in classical scheme theory; in particular, one constructs a Berkovich $K$-analytic space by gluing together the Berkovich spectra $\M(\A)$ of these $K$-affinoid algebras. %%% Affinoid Algebras give rise to affinoid spaces. Berkovich K-analytic spaces is defined by a triple (X, \calA, \tau), where  X is a loc. Haus. space, \tau is a collection of compact subsets of X, and \calA is an atlas sending each compact subset to a K-affinoid space. See Jonsson p.83 Lecture notes.
\end{example}

\section{Berkovich's Classification Theorem, Revisited}\label{sec:BerkClassThm} 
An organising theme of this section is the language of filters, which gives a transparent way of understanding how certain key notions in topology, logic and non-Archimedean geometry may interact.

We motivate our study by way of a biased historical overview. On the side of geometry, the fact that the points of a Berkovich spectrum\footnote{Here, we assume that $\A$ is strictly $K$-affinoid and $K$ has non-trivial valuation.} $\M(\A)$ may be characterised as ultrafilters was already known by the 1990s \cite[Remark 2.5.21]{BerkovichMonograph}; on the side of logic, the fact that the (complete) types over a model may also be characterised as ultrafilters was well understood by the 1960s \cite{Morley}, if not earlier. Yet it was only within the last 10 years that the two perspectives started to converge. Most notably, fixing a valued field $K$ of rank 1, Hrushovski and Loeser \cite[\S 14.1]{HruLoe} showed that the Berkovich analytification of any quasi-projective variety $V$ over $K$ can be described using the language of definable types (cf. item (iii) of Summary Theorem~\ref{sumthm:BERKOVICH}). The power of this logical perspective may be measured by the fact that the authors were able to establish many deep results (e.g. Theorem~\ref{thm:BerkLOCALCONTRACT}) that were inaccessible to previous methods (at least, without relying on stronger hypotheses e.g. smoothness).

This sets up our present investigation. The understanding that models of a propositional geometric theory can be characterised as completely prime filters is well known to topos theorists, although the connections with model-theoretic types appear to be undeveloped (but see Discussion~\ref{dis:typesMODELSfilters}). In this section, we follow the model theorist's cue and use point-free techniques to study the points of the Berkovich spectra $\M(K\{R^{-1}T\})$ from Example~\ref{ex:BerkPOWERseriesRING}. % The main surprise is that, unlike Berkovich's original classification result, the point-free approach works equally well for both the trivially and non-trivially valued fields. This indicates that the \emph{algebraic} hypothesis of $K$ being non-trivially valued is in fact a \emph{point-set} hypothesis, and is not essential to the underlying mathematics.

\subsection{Berkovich's Disc Theorem} We fix the following hypothesis for the rest of this section. %%% Until otherwise stated, we fix the following hypothesis.

\begin{hypothesis}\label{hyp:BerkovichCLASSIFICATION}\hfill 
	\begin{enumerate}[label=(\roman*)]
 \item $K$ is an algebraically closed field, complete with respect to a non-Archimedean norm $|\cdot|$. We emphasise that $|\cdot|$ need not be non-trivial. 
\item Abusing notation, we use $K$ to denote its underlying set of points.\footnote{This is all standard if one works classically. In point-set topology, one considers the points of a space as a set of elements (in contrast to point-free topology), and thus it is natural to denote the underlying set also as $K$. However, subtleties arise in point-free topology -- see Summary~\ref{sum:classicalVSgeometric} for more discussion.}
%		\item $K$ is a geometric field (= for all $k\in K$, either $k=0$ or $k$ is a unit). Abusing notation, we use $K$ to denote its underlying set of points.\footnote{\label{footnote:classHYPOTHESIS} This is all standard if one works classically. In point-set topology, one considers the points of a space as a set of elements (in contrast to point-free topology), and thus it is natural to denote the underlying set also as $K$. Similarly, it is classically trivial that any element of a field must either be 0 or a unit. However, in point-free topology, it is more natural to work \emph{geometrically} instead, i.e. using colimits and finite limits (see e.g. \cite{VickersPtfreePtwise}). In which case, we no longer get for free that all fields are geometric. As such, we give this property explicitly here just to maintain as much geometricity as possible.}

		\item $R$ denotes any positive Dedekind real $R>0$, and $Q_+$ denotes the set of positive rationals. 
		\item Define $K_R:=\{k\in K \, |\, |k|\leq R \}$.
		\item Following Example~\ref{ex:BerkPOWERseriesRING}: define a Banach ring $(\A,||\cdot||)$, where $\A=K\{R^{-1}T\}$ denotes the ring of power series converging on radius $R$, and $||\cdot||$ denotes the associated Gauss norm.
	\end{enumerate}
\end{hypothesis}

%\begin{observation} The Berkovich spectrum $\M(\A)$ is sometimes called the \textit{Berkovich disc of radius $r$}. \end{observation}

In order to classify the bounded multiplicative seminorms on $\A$, we first reduce our study to something algebraically simpler. 

\begin{definition}[Bounded $K$-Seminorms]\label{def:kseminorm} \hfill 
	\begin{enumerate}[label=(\roman*)]
		\item Define the space of linear polynomials $\Alin$  as
		\begin{equation}
		\Alin:=\{ aT-b \,|\, a,b \in K\}\cong K^2.
		\end{equation}
		In particular, by setting $a=0$, notice that $K\subset \Alin$.

		\item We define a \emph{$K$-seminorm} on $\Alin$ to be an upper-valued map
		\begin{equation}
		|\cdot|_x\colon \Alin\longrightarrow \overleftarrow{[0,\infty)}
		\end{equation}
		such that 
		\begin{itemize}
			\item \emph{(Preserves constants)} $|a|_x$ = the right Dedekind section of $|a|$
			% more formally: $\forall q\in Q_+$. $|a|_x<q \leftrightarrow |a|<q$
			\item \emph{(Semi-multiplicative)} $|aT|_x=|a|\cdot|T|_x$
			\item \emph{(Ultrametric Inequality)} $|f+f'|_x\leq \max\{|f|_x,|f'|_x\}$
		\end{itemize}
		for all $a\in K$, and $f,f'\in\Alin$.
		\item We define the \emph{Gauss Norm on $\Alin$} as
		\[||aT-b||:= \text{right Dedekind section of} \max\{|a|R, |b|\}\]
		where $aT-b\in\Alin$. We call a $K$-seminorm $|\cdot|_x$ is called  \emph{bounded} if $|\cdot|_x\leq ||\cdot||$.
	\end{enumerate}
\end{definition}

\begin{remark}\label{rem:Kseminorms} It is well-known \cite[Lemma 1.1]{BakerRumeley} that any bounded multiplicative seminorm $|\cdot|_x$ satisfies $|a|_x=|a|$ and $|f+g|_x\leq \max\{|f|,|g|\}$ for any $f,g\in\A$ --- this justifies the axioms we chose in Definition~\ref{def:kseminorm}(ii). Notice also that we did not require a $K$-seminorm to be multiplicative (only semi-multiplicative), but this is reasonable since $\Alin$ is not closed under multiplication.%\footnote{Notice: while there exists a natural additive structure on the linear polynomials, the product of two linear polynomials is no longer linear.}
\end{remark}

\begin{reminder} To eliminate potential confusion:
	\begin{itemize}
		\item A \emph{multiplicative seminorm} is defined on the whole ring $\A$ of convergent power series.
		\item A \emph{$K$-seminorm}, which is not multiplicative, is only defined on the space of linear polynomials $\Alin$.
	\end{itemize}
\end{reminder}

We justify the reduction to linear polynomials in the following Preparation Lemma.

\begin{lemma}[Preparation Lemma]\label{lem:PREPARATORY}\hfill 
	\begin{enumerate}[label=(\roman*)]
		\item Let $(\mathcal{B},||\cdot||_{\mathcal{B}})$ be an arbitrary Banach ring with dense subring $\mathcal{B}'$. Then, any bounded multiplicative seminorm on $\mathcal{B}$ is determined by its values on $\mathcal{B}'$.
		\item Any bounded multiplicative seminorm on $\A$ is determined by its values on linear polynomials $T-a$, where $a\in K_R$. The same also holds for bounded $K$-seminorms on $\Alin$.
	\end{enumerate}
\end{lemma}
\begin{proof} (i) Fix a bounded multiplicative seminorm $|\cdot|_x$ on $\mathcal{B}$, and suppose $f\in\mathcal{B}$. Next, for any positive rational $\epsilon>0$ and any $g\in\mathcal{B'}$ such that $||f-g||_{\mathcal{B}}<\epsilon$, compute:
	\begin{equation*}
	|f|_x\leq |g|_x + |f-g|_x\leq |g|_x + \epsilon,
	\end{equation*}
	\begin{equation*}
	|g|_x\leq |f|_x + |f-g|_x\leq |f|_x + \epsilon.
	\end{equation*}
	Hence, deduce that if $g\to f$ with respect to $||\cdot||_{\mathcal{B}}$, then $|g|_x\to |f|_x$. The claim then follows from $\mathcal{B}'$ being a dense subalgebra in $\mathcal{B}$.\newline  
	%%% In particular, note the equations give |f|_x\leq |g|_x +\epsilon \leq |f|_x + 2\epsilon. Observe what happens as $\epsilon \to 0$.
	
	\noindent (ii) The following two basic observations get us almost all the way:
	\begin{enumerate}[label=(\alph*)]
		\item The polynomial ring $K[T]$ is a dense subalgebra in $\A$, since any $f\in\A$ can be expressed as
		\[f=\sum_{i=0}^{\infty}a_iT^i = \lim_{n\to\infty}\sum^{n}_{i=0}a_i T^i.\]
		\item Since $K$ is algebraically closed, any polynomial $g\in K[T]$ can be expressed as 
		\[	g=c \cdot \prod^m_{j=1} (T-b_j),\]
		where $c,b_j\in K$, for all $1\leq j\leq m$. 
	\end{enumerate}
	Applying item (i) of the Lemma, Observations (a) and (b) together imply that any bounded multiplicative seminorm $|\cdot|_x$ on $\A$ is determined by its values on linear polynomials $T-b$ with $b\in K$. 
	
	In fact, a simple argument shows that we can restrict to just the linear polynomials $T-a$ where $a\in K_R$. Denote $h:=T-b$ to be a linear polynomial where $b\in K$. By polynomial division, we can represent the polynomial $T$ as 
	\[T = qh + z.\]
	Taking the Gauss norm $||\cdot||$ on both sides of the equation, we obtain:
	\[R=||T||=\max\{||qh||\,,\, ||z||\},\] %%% this is by definition of the Gauss norm.
	and so deduce $z\in K_R$. Since it is clear that $q$ must be a unit in $K$, we get
	\[q^{-1}\cdot (T-z) = h = T - b,\]
	and so our claim follows. The argument for the $K$-seminorm case is virtually identical.
\end{proof}

\begin{remark} For the reader familiar with non-Archimedean geometry: the proof of Preparation Lemma~\ref{lem:PREPARATORY} gives a prototype argument that can be generalised to prove the well-known Weierstrass Preparation Theorem \cite[Prop. 1.9.10]{JonssonWeierstrass}, which informally says: convergent power series often look like polynomials when restricted to a closed disc. Stated more precisely: all non-zero $f\in\A$ with finite order $m\geq 0$ can be represented as 
	\begin{equation}
	f= W\cdot u(T),
	\end{equation}
	where $u(T)$ is a unit power series on the closed disc $K_R$, and $W$ is a monic polynomial (the ``Weierstrass polynomial'') with both degree and order $m$. However, unlike e.g. \cite{BakerRumeley}, we shall avoid invoking the Weierstrass Preparation Theorem since it only applies to non-zero $f\in\A$ with \emph{finite order}. This will create issues if we wish to consider $K\{R^{-1}T\}$ with irrational radius $R$, which we can avoid with our approach.
\end{remark}

We now introduce a special class of filters that highlights the topological structure of $\M(\A)$.

\begin{definition}[Formal Non-Archimedean Balls]\label{def:formalballs} A \textit{formal non-Archimedean ball} is an element $(k,q)\in K_R\times \qp$. We represent this using the more suggestive notation $B_q(k)$, to emphasise that we should view this pair as denoting a disc of radius $q$ centred at $k$. In particular, we write:
	\[B_{q'}(k')\subseteq B_{q}(k) \, \,\text{just in case} \,\, |k-k'|<q \land q'\leq q.\]
\end{definition}

\begin{observation}\label{obs:BALLsequents} By decidability\footnote{That is, for any $q,r\in\qp$, we (constructively) know that $q<r$, $q=r$ or $q>r$.}  of $<$ on $\qp$, one obtains the following sequents from the definition of $\subseteq$ from Definition~\ref{def:formalballs}:
	\begin{enumerate}[label=(\roman*)]
		\item $|k-k'|<q\longrightarrow B_{q}(k')=B_{q}(k)$.
	%	\item $q\in \qp$ and $ |k-k'|<q'\longrightarrow B_{q}(k')=B_{q}(k)$ or $ q<q'$
		\item $B_{q}(k)=B_{q'}(k')\longrightarrow q=q'$
	\end{enumerate}
\end{observation}

\begin{discussion}[Rigid discs vs. Formal balls]\label{dis:pointSEThypotheses}  Notice item (ii) of Observation~\ref{obs:BALLsequents} says that the radii of the formal balls are well-defined (essentially by construction), even when the norm $|\cdot|$ on $K$ is trivial. This should be contrasted with the classical rigid discs from Example~\ref{ex:BerkPOWERseriesRING} 
	\begin{equation}
	D_{r}(k):=\{b\in K \,\big| \, |b-k|\leq r\},
	\end{equation}
	whose radii are well-defined only if $|\cdot|$ is non-trivial due to the point-set formulation. %(cf. Footnote~\ref{footnote:RADIIdisc}).
	%Say something about the tension between closed and opens balls, and how in the non-Archimedean case, we have a rich supply of clopens.. 
	
\end{discussion}
\begin{remark} The language of formal balls may strike the classical reader as a peculiar abstraction, but in fact they are sufficiently expressive to provide point-free accounts of standard completions of metric spaces \cite{ViLoccompI,ViLocCompII}. Once the connection between the Berkovich construction and the completion of a space is made precise, these techniques can be adjusted accordingly to our present context.\footnote{One important adjustment is that the \emph{non-strict} order defined in Definition~\ref{def:formalballs} is quite different from the \emph{strict} order defined in \cite{ViLoccompI,ViLocCompII}, but this reflects the fact that closed discs in non-Archimedean topology are also open.}
\end{remark}

\begin{definition}[Filters of Formal Non-Archimedean Balls]\label{def:filtersformalballs} A \emph{filter} of formal non-Archimedean balls $\F$ is an inhabited\footnote{The classical reader may substitute mentions of ``inhabited'' with ``non-empty'' without too much trouble.} subset of $K_R\times \qp$ satisfying the following conditions:
	\begin{itemize}
		\item \textit{(Upward closed with respect to $\subseteq$)} If $B_{q'}(k')\subseteq B_{q}(k)$ and $B_{q'}(k')\in \F$, then $B_{q}(k)\in \F$.
		\item \textit{(Closed under pairwise intersections)} If $B_{q}(k),B_{q'}(k')\in \F$, then there exists some $B_{r}(j)\in \F$ such that $B_r(j)\subseteq B_{q}(k)$ and $B_r(j)\subseteq B_{q'}(k')$. %For technical reasons, if $k,k'$ are non-zero, we shall require that $j$ in $B_{r}(j)$ be non-zero as well.
	\end{itemize}
	Further, we call $\F$ an \emph{$R$-good} filter if it also satisfies the following two conditions:
	\begin{itemize}
		\item For any $k\in K_R$, and $q\in Q_+$ such that $R<q$, $B_{q}(k)\in\mathcal{F}$. 
		\item  If $B_{q}(k)\in\F$, there exists $B_{q'}(k')\in \F$ such that $q'<q$.
	\end{itemize}
\end{definition}

\begin{convention} In this section, unless stated otherwise:
	\begin{itemize}
		\item A ball $B_q(k)$ will always mean a formal non-Archimedean ball (Definition~\ref{def:formalballs}).
		\item A filter $\F$ will always means a filter of formal non-Archimedean balls (Definition~\ref{def:filtersformalballs}).
	\end{itemize}
\end{convention}

\begin{comment}
\begin{observation} Any filter $\F$ is well-ordered by $\subseteq$, i.e. $\F$ is a nested descending sequence of balls.
\end{observation}
\end{comment}

\begin{observation}[Radius of $R$-Good Filters]\label{obs:FILTERradius} Let $\F$ be an $R$-good filter. Define the \emph{radius} $\rf$ of $\F$ as:
	\[\rf:=\{q\in Q _+|\, B_q(k)\in\F\,\}.\]
	Then, $\rf$ defines an upper real in $\overleftarrow{[0,R]}$.
\end{observation}
\begin{proof} Roundedness and upward closure is immediate from the definition of $\F$ being an $R$-good filter, and so $\rf$ defines an upper real. The fact that $\rf\in\overleftarrow{[0,R]}$ follows from additionally noting:
	\begin{itemize}
		\item $\rf$ is a subset of positive rationals. Hence, $0\leq \rf$.
		\item By $\calF$ being $R$-good, we know $q\in\rf$ for all rationals $q>R$. Hence, $\rf\leq R$.
	\end{itemize} 
\end{proof}

\begin{construction}\label{cons:SEMINORMTOFILTER} Suppose we have a bounded $K$-seminorm $|\cdot|_x$ on $\Alin$. We then define the following collection of formal balls:
	\[\F_{x}:=\{B_{q}(k) \, | \, k\in K_R\, \text{and}
	\, |T-k|_x<q\}\]
\end{construction}

\begin{claim}\label{claim:GETrgoodFILTER} $\F_x$ is an $R$-good filter.
\end{claim}
\begin{proof} By Definition~\ref{def:filtersformalballs}, we need to check that $\F_x$ is \dots 
	\begin{itemize}
		\item \emph{\dots Upward closed.} Suppose $B_{q'}(k')\subseteq B_{q}(k)$ and $B_{q'}(k')\in \F_x$. Unpacking definitions, this means $|k-k'|<q$ and $q'\leq q$, as well as $|T-k'|_x<q'$. But since
		\[|T-k|_x = |(T-k') + (k'-k)|_x \leq  \max\{|T-k'|_x,|k-k'|\}<\max\{q',q\}=q,\]
		this implies $B_{q}(k)\in \F_x$.\newline 
		
		\item \emph{\dots Closed under Pairwise Intersection.}  We first claim $\F_x$ is totally ordered by $\subseteq$. Why? Given any $B_{q}(k),B_{q'}(k')\in \F_x$, we get $|T-k|_x<q$ and $|T-k'|_x<q'$, and so
		\[|k-k'|=|(T-k')-(T-k)|_x \leq  \max\{|T-k'|_x,|T-k|_x\}<\max\{q',q\}.\]
		By decidability of $<$ on $\qp$, this means either $B_{q}(k)\subseteq B_{q'}(k')$ or $B_{q'}(k')\subseteq B_{q}(k)$, as claimed. The fact that $\F_x$ is closed under pairwise intersection follows immediately. \newline 
		%	Decidability of $<$ on $\qp$ means that either $q\leq q'$ or $q\geq q'$, which means that either $B_{q}(k)\subseteq B_{q'}(k')$ or $B_{q'}(k')\subseteq B_{q}(k)$, as claimed. The fact that $\F_x$ is closed under finite intersection follows immediately. 
		%Notice the technical requirement in the case when $k,k'$ are non-zero is also immediately satisfied. 

		\item \emph{\dots $R$-good.} 
		\begin{itemize}
			\item Suppose $B_q(k)\in\F_x$, and so $|T-k|_x<q$ by definition. Since $|\cdot|_x$ defines an upper real, there exists $q'\in\qp$ such that $|T-k|_x<q'<q$, and so $B_{q'}(k)\in \F_x$.
			\item Suppose $k\in K_R$ and $q\in\qp$ such that $R<q$. This gives
			\[|T-k|_x\leq \max\{|T|_x,|k|\}\leq \max\{||T||,|k|\}= R<q,\]
			since $k\in K_R$ implies $|k|\leq R$ by definition, and $|T|_x\leq ||T||=R$. Hence, $B_q(k)\in\F_x$.\newline 
		\end{itemize}

		\item \emph{\dots Inhabited.} Immediate from $R$-goodness.
	\end{itemize}
\end{proof}

In the converse direction, we define the following:

\begin{construction}\label{cons:FILTERTOSEMINORM} For any formal non-Archimedean ball $B_{q}(k)$, we define $|\cdot|_{B_{q}(k)}$ as follows:
	\[|T-a|_{B_{q}(k)}:=\max\{|k-a|, q\},\qquad \qquad \text{where $T-a\in\Alin$}.\]
	More generally, given an $R$-good filter $\F$, we define $|\cdot|_\F$ as
	\[|T-a|_{\F}:=\underset{B_{q}(k)\in \F}{\inf} |T-a|_{B_q(k)}=\underset{B_{q}(k)\in \F}{\inf}\max\{|k-a|,q\} \]	
	for any linear polynomial $T-a\in\Alin$.
\end{construction}

\begin{remark}\label{rem:AVOIDsups} Let $f\in\A$ such that $f$ converges on a rigid disc $D_r(k)$.  By the Maximum Modulus Principle in Non-Archimedean Analysis (see e.g. \cite[p. 3]{BakerRumeley}), one can express $|\cdot|_{D_{r}(k)}$ as 
		\begin{equation}\label{eq:discSEMINORM}
	|f|_{D_{r}(k)}=\sup_i|c_i|r^i\qquad\qquad \text{where $f=\sum_{i=0}^{\infty}c_i(T-k)^i$}.
	\end{equation}
Notice that $|T-a|_{D_q(k)}$ is classically equivalent to $|T-a|_{B_q(k)}$ just in case $r=q$. Nonetheless, the definition of $|\cdot|_{D_(r)}(k)$ as stated is problematic in our setting for two reasons.
	\begin{enumerate}[label=(\alph*)]
		\item First, the use of $\sup$ presents a constructive issue. A $\sup$ of Dedekinds yields a \emph{lower real} (see Definition~\ref{def:upperreal}), whereas our $K$-seminorms are valued in upper reals, and there is no constructive way to switch between lower and upper reals.\footnote{Why? See Discussion~\ref{dis:classASSBerk}. A possible objection: since $f\in\A$, $\sup_i$ is really a $\max_i$ because $|c_i|r_i\to 0$. Now since we can take finite $\max$'s of upper reals, does this not circumvent the constructivist issue? The answer is no: we do not know \emph{a priori} on which index $i$ does $\sup_i|c_i|r^i$ attains the maximum value. As such, we are still forced to quantify over all indices.
		} On the other hand, $\max$ and $\inf$ \textit{are} well-defined on the upper reals, which explains the formulation of Construction~\ref{cons:FILTERTOSEMINORM}. 
		\item The rigid discs featured in Equation~\eqref{eq:discSEMINORM} are required to have bounded radius $0<r\leq R$, which ensures the convergence of $|f|_{D_{r}(k)}$ for any $f\in\A$. On the other hand, the radius of the formal balls $B_{q}(k)$ have no upper bound. Nonetheless, we avoid convergence issues since $|\cdot|_{B_q(k)}$ is restricted to just the linear polynomials. 
	\end{enumerate}
	%\footnote{Why? Note that $T-a=T-k+k-a$}
	% (e.g. such that $(L,U)$ define a Dedekind real) 
	%%% But also, are the convergence issues? If q is bigger than R, this thing might not converge... 
	%%% Does Max Mod Principle work for trivially-valued norms? Yes, trivially. Everything is either 0 or 1.
\end{remark}

Having established Construction~\ref{cons:FILTERTOSEMINORM}, we perform the obligatory check:

\begin{claim}\label{claim:GETk-SEMINORM} $|\cdot|_{\F}$ determines a bounded $K$-seminorm on $\Alin$. More explicitly, for any $B_{q}(k)\in\F$, define
		\begin{align}\label{eq:EXTENDkseminorm}
	|\cdot|_{B_q(k)}\colon \Alin&\longrightarrow \mathbb{R}_{\geq 0}\\\nonumber
	aT-b&\longmapsto \max\{|ak-b|,|a|\cdot q\}.
	\end{align}
Then, the map 	
		\begin{equation}\label{eq:FILTERdefineSEMINORM}
	|aT-b|_\F:=\underset{B_q(k)\in \F}{\inf} |aT-b|_{B_q(k)}.
	\end{equation}
defines a bounded $K$-seminorm on $\Alin$.
\end{claim}

\begin{remark} It is easy to see that Equation~\eqref{eq:EXTENDkseminorm} recovers the original $|\cdot|_{B_q(k)}$ in Construction~\ref{cons:FILTERTOSEMINORM} when we restrict to $T-b$. In fact, if we work classically, then this extension of $|\cdot|_{B_q(k)}$ is canonical in that it is essentially forced upon us by Definition~\ref{def:kseminorm}.\footnote{Why? Classically, for any field $K$, it is decidable if $a\in K$ is a unit or $a=0$. (Geometrically, however, this is not true -- one has to impose this property as an extra axiom, see \cite{JohnstoneSpectra}.) Nonetheless, suppose $|\cdot|_x$ is a $K$-seminorm such that $|T-c|_x=$ ``the right Dedekind section of $|T-c|_{B_q(k)}$'', for any $c\in K$.  
If $a$ is a unit, then semi-multiplicativity then gives $|aT-b|_x=|a\cdot(T-b\cdot a^{-1})|_x=|a|\cdot |T-b\cdot a^{-1}|_{B_q(k)} = \max\{|ak-b|,|a|\cdot q\}$, whereas if $a=0$, then $|aT-b|_x=|b| = \max\{|ak-b|,|a|\cdot q\}$,	coinciding with (the right Dedekind section of) the extended $|\cdot|_{B_q(k)}$.}

\end{remark}

\begin{proof}[Proof of Claim] Verifying the claim involves checking that Equation~\eqref{eq:FILTERdefineSEMINORM} satisfies the required properties.
	\begin{itemize}
		\item $|\cdot|_\F$ is valued in the upper reals, since $|f|_\F$ takes the infimum of an arbitrary set of Dedekinds.
		\item To verify $|\cdot|_\F$ satsifies the properties listed in Definition~\ref{def:kseminorm}(ii), one first verifies their obvious analogues for $|\cdot|_{B_q(k)}$, before observing that they are preserved by taking $\inf$'s. 	
	Most are obvious except perhaps the ultrametric inequality. Suppose we have $aT-b,a'T-b'\in\Alin$. Then, given any $B_{q}(k)\in\F$, compute:
		\begin{align}\label{eq:BASICultrametricKSEMINORM}
		|aT - b + a'T - b'|_{B_q(k)} &= \max\{|(a+a')k-(b+b')|\,,\, |a+a'|\cdot q\} \nonumber \\
		&\leq \max\left\{ \max\{|ak-b|,|a'k-b'|\},\max\{|a|\cdot q,|a'|\cdot q\}\right\} \nonumber\\
		&= \max\{|aT-b|_{B_{q}(k)},|a'T-b'|_{B_q(k)}\}
		\end{align} 
		where the middle inequality is by the ultrametric inequality satisfied by the original norm $|\cdot|$ on $K$. Since this inequality holds for all $B_{q}(k)\in\F$, taking the infimum on both sides of the Inequality~\eqref{eq:BASICultrametricKSEMINORM} over all $B_{q}(k)\in\F$ gives the desired ultrametric inequality for $|\cdot|_\F$
		\begin{align}\label{eq:ultrametricKSEMINORM}
		|aT - b + a'T - b'|_{\F} \leq  \max\{|aT-b|_\F\,,\, |a'T-b'|_\F\}.
		\end{align}
		\item For any $R$-good filter $\F$,
		\begin{equation}
      |T-a|_\F=\underset{B_{q}(k)\in \F}{\inf}\max\{|k-a|,q\}\leq  \underset{B_{q}(k)\in \F}{\inf}\max\{|a|, |k|, q\}\leq \max\{|a|,R\},
		\end{equation}	
 where the final inequality is by $R$-goodness of $\F$ plus the fact that $k\in K_R$. The same argument extends to show that $|aT-b|_\F\leq \max\{|a|\cdot R,|b|\}$.  Hence, conclude that $|\cdot|_\F$ is bounded by the Gauss norm $||\cdot||$.\footnote{Technically, $\max\{|a|,R\}$ defines a Dedekind and not an upper real, so we really mean its corresponding right Dedekind section.}
	\end{itemize}
\end{proof}

As the reader may have anticipated, the algebraic constructions (i.e. bounded multiplicative seminorms and $K$-seminorms) and the topological constructions (i.e. the $R$-good filters) defined in this section have a close interaction. This is made precise in the following two theorems. %For clarity, we remind the reader that both theorems continue to assume Hypothesis~\ref{hyp:BerkovichCLASSIFICATION}. 

\begin{theorem}\label{thm:Alin=rgoodFILTER} The space of bounded $K$-seminorms on $\Alin$ is equivalent to the space of $R$-good filters.
\end{theorem}
\begin{proof} It suffices to show that Constructions~\ref{cons:SEMINORMTOFILTER} and \ref{cons:FILTERTOSEMINORM} are inverse to each other. This amounts to checking: 
	
	\subsubsection*{First Direction: $|\cdot|_x=|\cdot|_{\F_{x}}$} Fix a bounded  $K$-seminorm $|\cdot|_x$. By Preparation Lemma~\ref{lem:PREPARATORY}, it suffices to check that $|\cdot|$ and $|\cdot|_{\F_x}$ agree on linear polynomials $T-a$ such that $a\in K_R$. 
	
	Suppose $|T-a|_x<q$ for some $q\in\qp$. Then, $B_{q}(a)\in \F_x$ by construction. In particular, there exists 
	$q'\in Q_+$ such that
	\begin{equation}\label{eq:BqrNORM}
	|T-a|_{x}<q'<q
	\end{equation}
	since $|T-a|_x$ defines an upper real. Further, since
	\begin{equation}\label{eq:BqrNORM2}
	|T-a|_{B_{q'}(a)}=\max\{|a-a|,q'\}=q',
	\end{equation}
	and since Equation~\eqref{eq:BqrNORM} implies $B_{q'}(a)\in\F_x$, deduce that
	\begin{equation*}
	|T-a|_{\F_{x}}=\underset{B_{q}(k)\in \F_x}{\inf} |T-a|_{B_q(k)}\leq |T-a|_{B_{q'}(a)}<q,
	\end{equation*}
	and so 
	\begin{equation}\label{eq:Fxsmallerthanx}
	|T-a|_{\F_{x}}\leq |T-a|_x.
	\end{equation}
	
	Conversely, suppose $|T-a|_{\F_x}<q$ for some $q\in\qp$. Again, since $|T-a|_{\F_x}$ defines an upper real, deduce there exists $B_{q'}(k)\in\F_x$ such that 
	\[|T-a|_{\F_x}\leq |T-a|_{B_{q'}(k)}=\max\{|k-a|,q'\}<q.\]
	By definition, $B_{q'}(k)\in\F_{x}$ implies $|T-k|_x<q'$, and so 
	\[|T-a|_x=|(T-k)+(k-a)|_x\leq\max\{|T-k|_x,|k-a|\}<q,\]
	which in turn implies
	\begin{equation}\label{eq:xsmallerthanFx}
	|T-a|_x\leq |T-a|_{\F_{x}}.
	\end{equation}
	Put together, Equations~\eqref{eq:Fxsmallerthanx} and \eqref{eq:xsmallerthanFx} give $|\cdot|_x=|\cdot|_{\F_x}$, as claimed.\newline 
	
	\subsubsection*{Second Direction: $\mathcal{F}=\mathcal{F}_{|\cdot|_{\mathcal{F}}}$} Fix an $R$-good filter $\F$. Suppose $B_{q}(k)\in\F$. Since the radius $\rf$ of $\F$ defines an upper real, find a $q'\in\qp$ such that $B_{q'}(k')\in\F$ and $\rf<q'<q$. Without loss of generality, we may assume $k=k'$.\footnote{Why? Since $\F$ is closed under pairwise intersection, there exists $B_r(j)\subseteq B_{q}(k)\cap B_{q'}(k')$, and so by Observation~\ref{obs:BALLsequents}(i) $B_{q'}(k')=B_{q'}(j)\subseteq  B_q(j)=B_q(k)$.} Since $|T-k|_{B_{q'}(k)}=q'$, deduce that
	\[|T-k|_{\F}\leq |T-k|_{B_{q'}(k)}=q'<q.\]
	In particular, this implies $B_{q}(k)\in \F_{|\cdot|_\F}$, and so 
	\begin{equation}\label{eq:FinsideFNORM}
	\F\subset \F_{|\cdot|_\F}.
	\end{equation}
	
		Conversely, suppose $B_{q}(k)\in\F_{|\cdot|_\F}$. Unpacking definitions, deduce there exists $B_{q'}(j)\in\F$ such that 
	\begin{equation}\label{eq:beforeCASESPLIT}
	|T-k|_\F\leq |T-k|_{B_{q'}(j)}=\max\{|k-j|,q'\}<q.
	\end{equation}
	To show that $B_{q'}(j)\subseteq B_{q}(k)$, we need to show that $|k-j|<q$ and $q'\leq q$. But this is clear from Equation~\eqref{eq:beforeCASESPLIT}. 
	%But this is clear since Equation~\eqref{eq:beforeCASESPLIT} implies 
	%\begin{equation}
	%|k-j|\leq\max\{|k-j|,q'\}<q.
	%q'\leq \max\{|k-j|,q'\}<q.
	%\end{equation}	
	Since $B_{q'}(j)\in\F$ and $\F$ is upward closed, conclude that $B_{q}(k)\in\F$, and so
	\begin{equation}\label{eq:FNORMinsideF}
	\F_{|\cdot|_\F}\subset \F.
	\end{equation}
	Combining Equations~\eqref{eq:FinsideFNORM} and \eqref{eq:FNORMinsideF} gives $\F=\F_{|\cdot|_\F}$, finishing the proof.
\end{proof}

\begin{theorem}\label{thm:BerkovichDISC} As our setup, denote:
	\begin{itemize}
		\item $\M(\A)$ as the classical Berkovich spectrum (of Dedekind-valued multiplicative seminorms);
		%	\item $\overleftarrow{\M(\A)}$ as the space of upper-valued multiplicative seminorms on $\A$;
		\item $\overleftarrow{\M(\Alin)}$ as the space of bounded $K$-seminorms on $\Alin$.
	\end{itemize}
	\underline{Then}, 
	\begin{enumerate}[label=(\roman*)]
		\item Each $|\cdot|_x\in \M(\A)$ is uniquely represented by an $R$-good filter.
		\item $\M(\A)$ is \textit{classically} equivalent to $\overleftarrow{\M(\Alin)}$. In particular, it is \textit{classically} equivalent to the space of $R$-good filters.
	\end{enumerate}	
\end{theorem}

\begin{proof} It is clear $|\cdot|_x\in\M(\A)$ restricts to a bounded $K$-seminorm, which we denote
	\begin{equation}\label{eq:KsemiRESTRICT}
	|\cdot|_{x}\big|_{\Alin}\colon \Alin\longrightarrow \overleftarrow{[0,\infty)}.
	\end{equation}
By Fact~\ref{fact:UpperDed}, the map $x\mapsto R_x$ sending a Dedekind to its right Dedekind section is monic. As such,
	(i) follows immediately from Preparation Lemma~\ref{lem:PREPARATORY} and Theorem~\ref{thm:Alin=rgoodFILTER}.
	
	For (ii), we start by declaring the following classical assumption, which we will use freely in the remainder of the proof:
	\begin{itemize}
		\item[$(\star)$] Any upper real $\gamma_U$ besides $\infty$ or $-\infty$ can be canonically associated to a Dedekind real $\gamma$ whose right Dedekind section corresponds to $\gamma_U$. Similarly, any bounded lower real can also be canonically associated to its corresponding Dedekind real. 
	\end{itemize}
	The argument proceeds (again) by way of construction. By Equation~\eqref{eq:KsemiRESTRICT}, we already know how to send $|\cdot|_x\in\M(\A)$ to a bounded $K$-seminorm in $\overleftarrow{\M(\Alin)}$. In the converse direction, notice Theorem~\ref{thm:Alin=rgoodFILTER} helpfully allows us to explicitly characterise the generic bounded $K$-seminorm as $|\cdot|_{\F}\in\overleftarrow{\M(\Alin)}$, where $\F$ is the generic $R$-good filter. We extend this to a multiplicative seminorm on $\A$ in two steps. First, since any polynomial $f\in K[T]$ can be represented as 
	\[f=c \cdot \prod^m_{j=1} (T-b_j),\]
	$|\cdot|_\F$ naturally extends to a map on $K[T]$, which (abusing notation) we also represent as:
	\begin{align}\label{eq:polynomialEXT}
	|\cdot|_{\F}\colon K[T] &\longrightarrow [0,\infty) \\
	f &\longmapsto |c|\cdot \prod^m_{j=1} |T-b_j|_{\F}. \nonumber 
	\end{align}
	%A standard check verifies that Equation~\eqref{eq:polynomialEXT} defines a bounded multiplicative seminorm on $K[T]$. 
	Second, since any power series $f\in \A$ can be represented as limit of polynomials
	\begin{equation*}\label{eq:REPpowerseries}
	f=\sum_{i=0}^{\infty}a_iT^i = \lim_{n\to\infty}\sum^{n}_{i=0}a_i T^i,
	\end{equation*}
	we define the obvious extension of $|\cdot|_\F$ to $\A$:
	\begin{align}\label{eq:powerseriesEXT}
	\widetilde{|\cdot|_{\F}}\colon \A&\longrightarrow [0,\infty) \\
	f &\longmapsto \lim_{n\to \infty} |\sum^{n}_{i=0}a_i T^i|_{\F} \nonumber 
	\end{align}
	
	A few orienting remarks:
	\begin{enumerate}[label=(\alph*)]
		\item Notice the implicit use of Assumption $(\star)$ throughout, e.g. the original definition of $|\cdot|_\F$ was valued in the upper reals, but we extended it to a Dedekind-valued map on $K[T]$ in Equation~\eqref{eq:polynomialEXT}. % Similarly, although the supremum of a set of Dedekinds defines a lower real (cf. Fact~\ref{fact:UpperDed}), we have defined $\widetilde{|\cdot|_\F}$ in Equation~\eqref{eq:powerseriesEXT} as always being Dedekind-valued. 
		\item Assumption $(\star)$ also strengthens item (i): we can now view $\M(\A)$ as a subspace of $\overleftarrow{\M(\Alin)}$. %\footnote{To test our understanding: why couldn't we claim this when just working constructively? Recall from Fact~\ref{fact:UpperDed} that $x\mapsto R_x$ only defines a monic map, \emph{not} an embedding.}
		\item Suppose the construction $\widetilde{|\cdot|_{\F}}$ defines a bounded multiplicative seminorm on $\A$. Since Assumption $(\star)$ allows us to treat $K$-seminorms as if they were valued in Dedekinds, it becomes straightforward to check that the constructions $|\cdot|_x\big|_{\Alin}$ and $|\cdot|_{\F}$ define maps which are inverse to each other. In particular, the identity 	\begin{equation}
		|\cdot|_\F = \widetilde{|\cdot|_\F}\big|_{\Alin}
		\end{equation}
		is immediate by construction, whereas the identity  
		\begin{equation}
		|\cdot|_x= \widetilde{|\cdot|_x\big|_{\Alin}}
		\end{equation}
		follows from Preparation Lemma~\ref{lem:PREPARATORY} and checking the values on the linear polynomials. % Put together, this would give the claimed (classical) equivalence. %%% now you have to check the topology, no?
	\end{enumerate}
	As such, in order to prove the (classical) equivalence stated in the Theorem, it remains to verify that $\widetilde{|\cdot|_{\F}}$ is in fact a bounded multiplicative seminorm on $\A$. The check relies on standard arguments from non-Archimedean analysis; for details, see Appendix~\ref{Appendix:BerkDisc}. 
\end{proof}

We conclude with some discussions on various aspects of the proof.

%%% \ref{rem:AVOIDsups}
\begin{discussion}\label{dis:PROOFofBerk} Let us sketch the original argument from classical Berkovich geometry.
	\begin{itemize}
		\item Given any rigid disc 
		$D_{r}(k)$ such that $0<r\leq R$, define a multiplicative seminorm $ |\cdot|_{D_{r}(k)}$ on $\A$. 
		\item Next, given any nested sequence of discs $\mathsf{D}:= D_{r_1}(k_1)\supseteq D_{r_2}(k_2)\supseteq \dots$, define the multiplicative seminorm $|\cdot|_{\mathsf{D}}:=\inf_{\mathsf{D}} |\cdot|_{D_{r_i}(k_i)}$. 
		\item Finally, given $|\cdot|_x\in\M(\A)$, define the nested sequence $\mathsf{D}_{x}:=\{D_{|T-k|_x}(k)\,|\, k\in K \, \text{and}\, |k|\leq R\}$. Check that $|\cdot|_x$ and $|\cdot|_{\mathsf{D}_x}$ agree on linear polynomials, and conclude $|\cdot|_x=|\cdot|_{\mathsf{D}_x}$. 
	\end{itemize}
	The parallels with the proof for Theorem~\ref{thm:Alin=rgoodFILTER} are clear. However, the argument must be adjusted and finitised appropriately in order to work in our context. Some important differences:
	%%% exchanged one finitisation for another. %%% pulling away from the set theory solves the problem of insufficient discs
	\begin{enumerate}[label=(\roman*)]
		\item \emph{On ``rational'' discs.} Both $R$-good filters and the nested sequence of discs $\mathsf{D}_x$ give rise to approximation arguments, but their approximants differ in important ways. In particular, whereas the radius of a formal ball $B_{q}(k)$ is rational in the usual sense that $q\in\qp$ is a positive rational number, in non-Archimedean geometry a rigid disc $D_r(k)\in \mathsf{D}_x$ with rational radius means that
		\[r\in\Gamma:=\{|k|\in [0,\infty) \,|\, k\in K\},\] 
		i.e. $r$ belongs to the value group $\Gamma$ of $K$. 
		The suggestive terminology (``rational'') indicates an analogy between $\qp$ and $\Gamma$, but it is important to remember that they are not the same --- particularly when $K$ is trivially-valued.
		\item \emph{On $K$-seminorms}. Whereas the original argument starts by defining a multiplicative seminorm on $\A$, before restricting it to the linear polynomials to perform certain checks, we instead defined a new algebraic structure (which we called $K$-seminorms) on the space of linear polynomials $\Alin$. 
		\item \emph{On the use of filters.} While Berkovich's original argument shows that every $|\cdot|_x\in\M(\A)$ corresponds to a nested descending sequence of discs, this representation is not unique. In particular, two different sequences of discs may define the same multiplicative seminorm on $\A$. We resolve this issue by appealing to the more natural language of filters, which allows us to obtain a representation result: every $|\cdot|_x\in\M(\A)$ is uniquely associated to an $R$-good filter $\F_x$. % This should be compared with our refinement of Ostrowski's Theorem in \cite{NVOstrowski}.
		\item \emph{On the use of formal balls.} As already pointed out in Discussion~\ref{dis:pointSEThypotheses}, our result holds for both trivially and non-trivially valued $K$. This is in contrast to the original argument, which only works for non-trivially valued $K$. % In particular, notice that $\Gamma=\{0,1\}$ when $K$ is trivially valued, which is clearly insufficient to approximate the real interval.
	\end{enumerate}
	%\footnote{As a sanity check: how were the properties of the $K$-seminorms used in our proof? The ultrametric inequality and preservation of constants were used extensively, especially when it came to splitting up equations of the form $|T-a|_x=|(T-k)+(k-a)|_x$. Semi-multiplicativity was used in Claim~\ref{claim:GETk-SEMINORM} to show that Construction~\ref{cons:FILTERTOSEMINORM} could be canonically extended to all of $\Alin$.}
	Items (i) and (ii) reflect our decision to work with the upper reals as opposed to the Dedekinds, and strike a careful balance: whilst the upper reals are particularly suited to analysing the filters of formal balls (cf. Observation~\ref{obs:FILTERradius}), they also impose strong (constructive) restrictions on the algebra (cf. Remark~\ref{rem:AVOIDsups}, but see also Discussion~\ref{dis:classASSBerk}). Items (iii) and (iv) give evidence that filters (as opposed to nested sequences of discs) and formal balls (as opposed to the classical rigid discs) are the correct language for studying $\M(\A)$.
\end{discussion}

%The reason why these structural arguments show up in our context (and not in Berkovich's original proof classifying the points of $\M(K\{R^{-1}T\})$ for fixed $R$) once again reflects 
%and that analogues of these claims do not appear in the original proof classifying the points of $\M(K\{R^{-1}T\})$ for fixed $R$. Once again, this adjustment reflects the restrictions imposed by the upper reals, as already discussed above.
%%% forces us to work more closely with K[T]

Finally, for the reader interested in constructive mathematics, we sort out and summarise the classical vs. constructive aspects of our result.

\begin{discussion}[On Assumption $(\star)$]\label{dis:classASSBerk} Recall that Theorem~\ref{thm:BerkovichDISC} (ii) is a classical result because the proof relies on Assumption $(\star)$, which essentially says: any bounded upper (or lower) real can be canonically associated to a Dedekind real. Some natural questions (and answers): 
	\begin{enumerate}[label=(\roman*)]
		\item \emph{Why is Assumption $(\star)$ classical?} Consider the obvious argument: given an upper real $R$ defined as the set of rationals strictly greater than 1, we define a lower real $L$ as the set of rationals strictly less than 1. Hence, conclude that $(L,R)$ is the Dedekind real canonically associated to $R$.
		
		This argument will strike most as reasonable, so where does it fail constructively? Discussion~\ref{dis:POVonUPPERS} reminds us that an upper real is blind to the rationals less than itself, suggesting it may only have the same knowledge as a Dedekind real by way of classical reasoning. Reformulated more precisely, we claim that if any upper real $R$ can be associated to a lower real $L$ such that $(L,R)$ defines a Dedekind real, then this implies that every proposition $p$ has a Boolean complement $p'$ --- which holds classically, but \textit{not} constructively (cf. Remark~\ref{rem:filtersPTS}).
		
		To prove our claim, let $p$ be any proposition, and define the subset of rationals:
		\[R:=\{ q\in\Q |\, \text{either ``$q>1$'' or ``$p$ holds and $q>0$''}\}\]
		In other words, $R$ is a kind of schizophrenic upper real: since $p$ holds iff $R<1$, $R$ may define the upper real $1$ or the upper real $0$, depending on the truth value of $p$. Now, suppose we have some lower real $L$ such that $(L,R)$ is Dedekind. Define a new proposition $p'$ iff $ \frac{1}{2}<L$. If $(L,R)$ indeed define a Dedekind, one deduces
		\begin{enumerate}[label=(\alph*)]
			\item $\top\to p\vee p'$;
			
			[Why? Locatedness of $(L,R)$ gives $\frac{1}{2}<L\vee R<1$.]
			\item  $p\land p'\to \bot$.
			
			[Why? Since $p\to R<\frac{1}{2}$, this gives $p\land p'\to (\frac{1}{2}<L)\land (R<\frac{1}{2})$, contradicting separatedness of $(L,R)$.]
		\end{enumerate}
	This shows $p'$ is a Boolean complement of $p$, as claimed. %%% See also Bas Spitters p. 5 Example 3 Locatedness and overt sublocales
		\item \emph{Is Assumption $(\star)$ necessary?} To constructivise Theorem~\ref{thm:BerkovichDISC}, one might look to prove an equivalence between upper-valued multiplicative seminorms on $\A$ and $K$-seminorms on $\Alin$ instead. If such a result does hold, we will need a different proof strategy. This seems challenging: Assumption $(\star)$ allows us to assume that bounded suprema/infima of Dedekinds are still Dedekinds, which plays a key (if implicit) role in our proof (see Claim~\ref{claim:POWERseminorm}). Notice that we 
 will need Assumption $(\star)$ anyway if we are to obtain the final equivalence with the usual Berkovich spectra. 
	\end{enumerate}
\end{discussion}

\begin{summary}[Classical vs. Geometric Mathematics]\label{sum:classicalVSgeometric} 
\emph{Geometric mathematics} is a specific regime of constructive mathematics, natural in point-free topology, which essentially deals with constructions arising from colimits and finite limits (for details, see \cite{Vi4,VickersPtfreePtwise}). Since our work here is influenced by point-free techniques, a natural question to ask is: to what extent do the results of this section adhere to these constraints? 
	\begin{enumerate}[label=(\roman*)]
		\item Hypothesis~\ref{hyp:BerkovichCLASSIFICATION} views $K$ as a set of elements, equipped with structure of a field. Taken at face value, these are unwelcome restrictions. If one wished to be consistently geometric throughout, this means $K$ has to be a discrete field and so many (topological) fields of interest, e.g. the $p$-adic complex numbers $\C_p$, would be excluded. However, if the reader is happy to work with point-set topology, then this is no longer a problem. %(cf. also Remark~\ref{rem:geomfield}).% \footnote{Although we haven't checked, it is reasonable to suspect that there may be constructive workaround so long as we are careful about how we rely on non-geometric properties. In practice, the generic model of a theory may satisfy certain non-geometric sequents which can be used in geometric proofs (so long as the final result can be stated geometrically) --- a textbook example is the generic model of the theory of local rings, which satisfies a non-geometric formulation of the field axiom. For more details on non-geometric field axioms, see \cite[\S 2]{JohnstoneSpectra}. For a more general perspective on the role of non-geometric sequents in applied topos theory, we recommend \cite{Blechschmidt}.} 
		\item Although the polynomial ring $K[T]$ and $\Alin$ can be defined geometrically (assuming Hypothesis~\ref{hyp:BerkovichCLASSIFICATION}), it is presently unclear how to do the same for the convergent power series ring $K\{R^{-1}T\}$ --- in particular, how to formulate the condition ``$f\in\A$ converges on a radius $R$'' geometrically.%\footnote{We have yet to check the details, but here's a plausible constructive workaround for items (i) and (ii): view $K\{R^{-1}T\}$ as a completion of some polynomial ring $K_0[T]$ with respect to the Gauss norm $||\cdot||$, where $K_0$ is an algebraically closed (discrete) field that can be completed with respect to the non-Archimedean norm $|\cdot|$ to get $K$. Unclear how algebraic closure or completion interact in this setting though.}  
		\item So long as Hypothesis~\ref{hyp:BerkovichCLASSIFICATION} holds (in fact, we can eliminate the assumption that $K$ is complete here), Theorem~\ref{thm:Alin=rgoodFILTER} is a fully geometric result. On the other hand, as pointed out in Discussion~\ref{dis:classASSBerk}, the full strength of  Theorem~\ref{thm:BerkovichDISC} is a classical result due to the use of Assumption $(\star)$.
	\end{enumerate}	
\end{summary}

%% For the reader committed to constructive/geometric mathematics: although we haven't checked the details, there may be a constructive workaround for items (i) and (ii) in Summary~\ref{sum:classicalVSgeometric} by viewing $K\{R^{-1}T\}$ as a completion of the polynomial ring $K_0[T]$ with respect to the Gauss norm $||\cdot||$, where $K_0$ is a geometric field that can be completed with respect to the non-Archimedean norm $|\cdot|$ to get $K$.

\subsection{Applications to the Trivial Case} As a slick application of Theorem~\ref{thm:BerkovichDISC}, we recover the familiar characterisations of $\M(\A)$ when $K$ is trivially valued (see Example~\ref{ex:BerkPOWERseriesRING}). What's new here? For one, the proofs given here appear quite different from the standard arguments \cite{JonssonAnnotations}, and are also shorter. More fundamentally, they give a very interesting indication of how Berkovich's characterisation of $\M(\A)$ (via nested sequences of discs) is in fact more robust than previously thought.

\begin{example}[Case: $R<1$]\label{ex:R<1} 
	If $R<1$ and $K$ is trivially valued, then $K_R=\{0\}$. The $R$-good filters are thus entirely determined by their radii, and so the space of $R$-good filters is equivalent to $\overleftarrow{[0,R]}$ (see Observation~\ref{obs:FILTERradius}). Applying Theorem~\ref{thm:BerkovichDISC}, one deduces that  $\M(\A)$ is classically equivalent to $\overleftarrow{[0,R]}$, essentially for free.
	
	\begin{figure}[h!]
	\centering
	\includegraphics[width=0.36\linewidth]{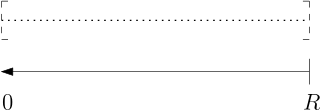}
	\caption{$\M(\A)$, when $K$ is trivially valued and $R<1$}
\end{figure}
\end{example}

\begin{example}[Case: $R\geq 1$]\label{ex:R>1} If $R\geq 1$ and $K$ is trivially valued, then notice $K_R=K$. Consider the following two subcases:
	\begin{itemize}
		\item[] \textbf{Subcase 1:} \textit{$\F$ is an $R$-good filter with radius $\rf\geq 1$.} In which case, $B_{q}(k)\in\F$ for any $k\in K$ and any $q>\rf$. [Why? Since $\F$ is $R$-good, there must exist some $B_{q}(k')\in\F$ for any $q>\rf$. Since $K$ is trivially valued, we get $|k-k'|\leq1 \leq\rf < q$ for any $k\in K$, and so $B_{q}(k')= B_{q}(k)$.]
		
		Hence, an $R$-good filter (with radius $\rf\geq 1$) is entirely determined by its radius, and so the space of such $R$-good filters form an interval $\overleftarrow{[1,R]}$.
		\item[] \textbf{Subcase 2:} \textit{$\F$ is an $R$-good filter with radius $\rf<1$.} In which case:
		\begin{enumerate}[label=(\alph*)]
			\item If $B_q(k)$ such that $q>1$, then $B_{q}(k)\in\F$.
			
			[Why? Since $\rf<1$, there must exist $B_{q'}(k')\in\F$ such that $q'\leq 1 <q$. Since $K$ is trivially valued, deduce that $B_{q'}(k')\subseteq B_{q}(k)$ and so $B_{q}(k)\in\F$ as $\F$ is upward closed.]
			
		%	\item If $B_q(0)\in\F$ and $q\leq 1$, then $k=0$ for any $B_{q'}(k)\in\F$ such that $q'\leq 1$.
			
		%	[Why? Take the intersection of $B_{q'}(k),B_{q}(0)\in\F$ to get $B_{q''}(j)\in\F$. Since $B_{q''}(j)\subseteq B_{q}(0)$, this forces $j=0$ since otherwise $1=|j-0|<q\leq 1$, contradiction. Further, since this implies $B_{q''}(0)\subseteq B_{q'}(k)$, deduce that $k=0$ for the same reason.]
			
			\item If $B_{q}(k)\in \F$ and $q\leq 1$, then $k=k'$ for any $B_{q'}(k')\in\F$ such that $q'\leq 1$.
		
			[Why? Take the ``pairwise intersection'' of  $B_q(k),B_{q'}(k')\in\F$ to get $B_{q''}(k'')\in\F$. Since $B_{q''}(k'')\subseteq B_{q}(k)$, this forces $k''=k$ since otherwise $1=|k''-k|< q\leq 1$, contradiction. The same argument shows that $k''=k'$, and so we conclude $k'=k$.]  
		\end{enumerate}
	\end{itemize}
	Summarising, an $R$-good filter $\F$ with radius $\rf\geq 1$ is entirely determined by its radius (Subcase 1), whereas an $R$-good filter $\F$ with $\rf< 1$ is determined by its radius plus its unique choice of $k\in K$ (Subcase 2). Applying Theorem~\ref{thm:BerkovichDISC} once more, this gives the following diagram on the left:
	\begin{figure}[ht!]
		\centering
		\subfloat[{{$\M(\A)$, when $K$ is trivially valued and $R\geq 1$}}]{	\includegraphics[width=0.47\linewidth]{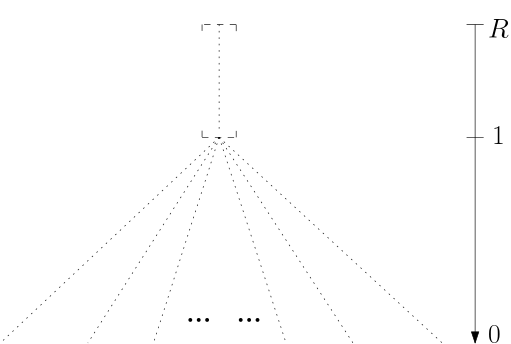}
			\label{fig:TrivialCase2}}\quad
		\subfloat[{{$[\protect\lav]$}}]{	\includegraphics[width=0.47\linewidth]{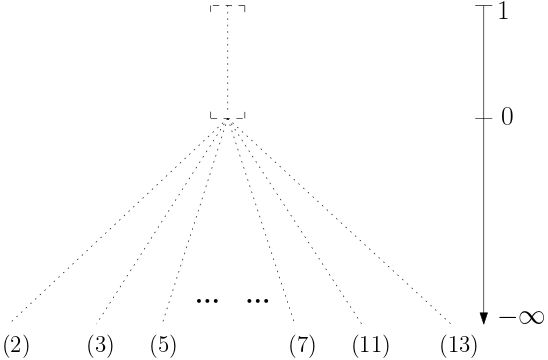}
			\label{fig:av}}	
		\caption{}
	\end{figure}
	
$\M(\A)$ should be compared with $[\lav]$, the space of multiplicative seminorms on $\Z$ defined in \cite{NVOstrowski}, pictured on the right. In particular, notice the key branching points of the spaces are identical up to a $\log_R\argu$ transformation.
\end{example}

\begin{remark} Of course, Examples~\ref{ex:R<1} and ~\ref{ex:R>1} present $\M(\A)$ using upper reals instead of Dedekinds (in particular, they are not Hausdorff), but this can be resolved by applying Assumption $(\star)$ once more. 
\end{remark}

\begin{comment}

\begin{discussion} Instructive to compare Alternative proof from \cite{JonssonAnnotations}. Involved...
\begin{enumerate}[label=(\roman*)]
\item $|\cdot|\mapsto |T|$
\end{enumerate}    
\end{discussion}

\end{comment}

\section{An Algebraic Fork in the Road}\label{sec:algforkinroad}

Let us review our work. In principle, the extension of Berkovich's original result to Theorem~\ref{thm:BerkovichDISC} could have been discovered much earlier. And yet it was not --- to our knowledge, the idea that one could modify the language of rigid discs to classify the points of $\M(K\{R^{-1}T\})$ \emph{without} requiring $K$ to be non-trivially valued was not suspected by the experts.\footnote{Difficult, of course, to properly gauge what the experts may or may not have suspected, but this may be inferred from how the literature emphasises the necessity of being non-trivially normed. For instance, in Jonsson's lecture notes on Berkovich's classification of the points of the Berkovich Affine line $\AffBerk^1$ over $K$, he remarks: ``The second assumption [that $K$ is non-trivially valued] is necessary [...] if the norm on $K$ is trivial, then there are too few discs.'' \cite[Proof of Theorem 3.10]{JonssonDynamicsLowDim}} The reason for this seems to be that Theorem~\ref{thm:BerkovichDISC}, both in its formulation and proof, belongs to the point-free perspective in an essential way. Of course, one certainly does not need to be a topos theorist in order to e.g. understand what a formal ball is, but there are specific intuitions from the point-free perspective that guided us to our result:
\begin{enumerate}[label=(\alph*)]
	\item The topos theorist is trained to recognise how the same idea may be expressed in different settings, and to ask about the connections.\footnote{For the insider: the topos theorist knows that there are many sketches of the same elephant \cite{J1,J2}. %See \cite{J1,J2}.
	} In particular, recall: any essentially propositional theory corresponds to a space of completely prime filters. Work in \cite{NVOstrowski} showed that the theory of multiplicative seminorms on $\Z$ is essentially propositional. Given the obvious parallels between $\M(\Z)$ and $\M(K\{R^{-1}T\})$ when $K$ is trivially normed and $R\geq 1$ (cf. Example~\ref{ex:R>1}), the topos theorist may guess that $\M(K\{R^{-1}T\})$ can also be described via completely prime filters.
	\item Once we defined the formal ball $B_{q}(k)$ and the correct inclusion relation $B_{q'}(k')\subseteq B_{q}(k)$, the rest of the argument began to fall into place. But notice: the decision to use formal balls (as opposed to the classical rigid discs) reflects the localic perspective that it is the \emph{opens} that are the basic units for defining a space, and \textbf{not} the underlying \emph{set of points}. Compare, for instance, our use of formal balls with the symbols $P_{qr}$ in the propositional theory of Dedekinds $\thT_\R$. 
\end{enumerate}

Though a relatively simple result, the surprising aspects of Theorem~\ref{thm:BerkovichDISC} hints at the clarifying potential of using point-free ideas to investigate foundational issues in non-Archimedean geometry. Motivated by this, we conclude with a list of inter-related problems, geared towards testing this idea.

\subsection{Trivially vs. Non-Trivially valued Fields} First, an obvious piece of mathematical due diligence. Many results in Berkovich geometry are sensitive to the case-split between trivially vs. non-trivially valued (non-Archimedean) fields. This motivates the following general exercise:

%%% an interesting theme of 

\begin{problem} Pick an interesting result in Berkovich geometry that appears to rely on the base field $K$ being non-trivially valued. Examine why. Just as in Theorem~\ref{thm:BerkovichDISC}, 
	can we eliminate this hypothesis by applying point-free techniques? If yes, what applications does this generalised result give us?
\end{problem}

\begin{discussion}\label{dis:rationaldiscsAPPROX} Here's one place to start looking. In non-Archimedean geometry, a common strategy for proving results on an irrational (or open) disc $\mathsf{D}$ is to first prove the result for \emph{rational closed discs}, before extending the result to $\mathsf{D}$ by expressing it as a nested union of rational closed discs (see e.g. \cite{Benedetto}). This strategy obviously breaks down when $K$ is trivially valued, but the language of upper reals may offer a workaround (cf. Discussion~\ref{dis:PROOFofBerk}).
\end{discussion}

\subsection{Overconvergent Lattices and Rigid Geometry} We now discuss a more solid lead. In unpublished work of Dudzik \cite{DudzikOverconvergent} as well as Baker's Berkeley Lecture Notes \cite{BakerBerkeley}, the following notion was defined:

\begin{definition} Consider a lower-bound distributive lattice $L$, with finite $\land$ and $\vee$ and a minimal element denoted $\bot$. 
	
	\begin{enumerate}[label=(\roman*)]
		\item For $x,x'\in L$, we say \emph{$x$ is inner in $x'$}, written $x\triangleleft x'$, if for all $z\geq x'$, there exists $w$ with $x\land w = \bot$ and $x'\vee w=z$.
		\item We call $L$ an \emph{overconvergent lattice} if for all $x,y\in L$ where $x\land y=\bot$, there exists $x'\in L$ such that $x\triangleleft x'$ and $x'\land y=\bot$. 
	\end{enumerate} 
\end{definition}

It was then showed that these so-called overconvergent lattices captured some key topological features of rigid analytic geometry. We summarise some of their key results in the following theorem:

\begin{theorem}\label{thm:Dudzik} As our setup,
	\begin{itemize}
		\item Let $\A$ be a (strict) affinoid algebra over a suitable\footnote{By which we mean: complete, non-trivially valued and non-Archimedean. Compare this with \cite[Remark 2.5.21]{BerkovichMonograph}, which we briefly discussed at the start of Section~\ref{sec:BerkClassThm}.} field $K$;
		\item Let $X=\Sp(\A)$, i.e. the set of maximal ideals of $\A$ equipped with the canonical topology;
		\item Let $L$ be the lattice of special subdomains of $X$;
		\item Let $\mathsf{P}(L)$ be the set of prime filters and $\mathsf{M}(L)$ the set of maximal filters.
	\end{itemize}
	\underline{Then:}
	\begin{enumerate}[label=(\roman*)]
		\item The lattice $L$ is overconvergent and its elements form a neighbourhood base. %%% Sarah Brodsky and Melody Chan
		\item There exists a canonical surjective map $\mathsf{P}(L)\to\mathsf{M}(L)$ sending a prime filter to the unique maximal filter containing it. When equipped with the quotient topology, $\mathsf{M}(L)$ is equivalent to the Berkovich analytification $X^{\an}$.  %%% Sarah Brodsky and Melody Chan
		\item  $\mathsf{P}(L)$ is equivalent to Huber's adic space of continuous semivaluations on $\A$. %%% March 8, Michael Daub. The connection with the lattice theoretic perspective is more implicit.
	\end{enumerate}	
\end{theorem}
%%% Start with Sp(A) and define G-topology on it via special subdomains (finite union of rational subdomains)
Closer examination of the mechanics underlying the  proof of Theorem~\ref{thm:Dudzik} seems warranted. Although the theorem is essentially a reworking of classical facts about rigid geometry \cite{FresnelvanderPut,vdPutSchneider}, the lattice-theoretic perspective 
brings into focus the key topological ingredients. In particular, locale theorists may recognise the family resemblance between overconvergent lattices and \emph{normal} lattices, which suggests that the hypothesis of overconvergence was chosen precisely to guarantee that each prime filter is contained in a unique maximal filter\footnote{Normality and overconvergence appear to be essentially dual notions. Recall: for a (bounded) distributive lattice $L$ and $a,a'\in L$, we say \emph{$a'$ is well inside $a$}, written $a'\eqslantless a$ if there exists $y$ such that $a\vee y=\top$ and $a'\land y=\bot$. Such a lattice $L$ is said to be \emph{normal} if whenever $a\vee b=\top$, there exists $a'\eqslantless a$ with $a'\vee b=\top$. In particular, a (bounded) distributive lattice $L$ is normal iff each prime ideal in $L$ is contained in a unique maximal ideal \cite[\S 3.6-3.7]{StoneSp}. Of course, this is a characterisation of normality for \emph{bounded} distributive lattices, but a similar (dual) result appears under the weaker hypothesis of only being lower-bounded \cite{Cornish}.} (see item (ii) of the Theorem). % In addition, Dudzik \cite{DudzikOverconvergent} also uses ideas from Hochster's theory of spectral spaces, except he extends it to a slightly different class of spaces known as \emph{quasi-spectral spaces} (= inverse limits of locally finite $T_0$ spaces).\footnote{In particular, $\mathsf{P}(L)$ is a quasi-spectral space.} 
%% but a similar (dual) result appears under the weaker hypothesis of only being lower-bounded \cite{Cornish}. The comment on duality comes from Johnstone's Stone Spaces, essentially Cornish looks at lattices of closed subsets not open subsets.
Some natural test problems and questions:

\begin{problem} In his note \cite{DudzikOverconvergent}, Dudzik left unfinished the problem of applying overconvergent lattices to the classification of the points of $\AffBerk^1$. A good exercise: finish this. In particular, our proof of Theorem~\ref{thm:BerkovichDISC} should be relevant. However, what do $R$-good filters have to do with the filters of overconvergent lattices? In addition, do point-free techniques allow us (once more) to eliminate the requirement that $K$ be non-trivially valued in Theorem~\ref{thm:Dudzik}?
\end{problem}

%\begin{problem} Develop and make precise the connections between overconvergent lattices and normal lattices. For instance, quite intriguingly, normal lattices have independently shown up in topos-theoretic approaches to quantum theory (see e.g. \cite{SpittersVickersWolters}) --- are there productive parallels that can be drawn between this and the role of overconvergent latices in non-Archimedean geometry? %% What may e.g. Wallman compactification have to say about Berkovich or adic spaces?
% \end{problem}

\subsection{Model-theoretic vs. Point-free perspectives} Interspersed throughout this chapter were various mentions of Hrushovski-Loeser's groundbreaking work \cite{HruLoe}, which applied model-theoretic tools to Berkovich geometry. For the topos theorist, a natural question is the following:

\begin{problem}\label{prob:HruLoePTFREE} Leveraging point-free techniques, simplify and/or extend the framework of Hrushovski-Loeser spaces.
\end{problem}

\begin{discussion} A good place to start: where do overconvergent lattices feature in their framework? Relatedly, the technology of pro-definable sets bears a strong resemblance to $R$-structures and rounded ideal completions (which implicitly featured in our use of upper reals). It would be interesting to see if this connection can be developed to indicate some natural simplifications. (Discussion~\ref{dis:rationaldiscsAPPROX} may be relevant.)
\end{discussion}

\begin{discussion}[The Role of Stability]\label{dis:HruLoeComplexity} Recall from Discussion~\ref{dis:LogicalComplexity} that the topos theorist and model theorist hold different attitudes towards the presence of strict order in a theory. For the model theorist, strict order is a sign that the theory is complex [i.e. unstable], and so a great deal of work is needed in order to account for e.g. the instability in $\ACVF$ via the language of stably-dominated types \cite{HruStableDomination,HruLoe}. On the other hand, strict order is \emph{not} a sign of complexity for the topos theorist --- the theory of the rationals and the Dedekinds both have strict order and are considered well-behaved [i.e. they are essentially propositional]. One is therefore led to ask: are (model-theoretic) notions such as stable/meta-stable/unstable truly essential to understanding the topology of non-Archimedean spaces? 
	
One way to parse this question: perhaps the way to reduce our sensitivity to strict order is not by passing to stably-dominated types but by passing to a different fragment of logic. Of course, given the context of this paper, we have geometric logic in mind but one may also wish to consider e.g. continuous logic, as done in \cite{BenYaacovACMVF}.\footnote{Notice, however, stability still plays a role in \cite{BenYaacovACMVF}, albeit in the continuous logic setting. Notice also that the paper restricts itself to \emph{metric valued fields} and not valued fields whose value group $\Gamma$ is an arbitrary ordered abelian group (with minimal element $0$). This will become relevant when we start thinking about extending our logical methods to adic spaces -- see also Discussion~\ref{dis:adic}.} Another approach is to ask: is the notion of stability itself fundamental to understanding how the residue field and value group interact? We make no firm assertions at this juncture, but recent work on residue field domination \cite{ResDom} indicates some interesting evidence to the contrary.

% 	in particular that the key questions in valued fields involve understanding how the residue field and value group interact, which goes beyond the usual  % (applicable to valued fields whose residue field sort is not necessarily stable)
\end{discussion}

\begin{discussion}[Adic Spaces]\label{dis:adic} There have been recent efforts to extend Hrushovski-Loeser's work to the setting of adic spaces (see e.g. \cite{JinhePabloProdefinable}). However, Theorem~\ref{thm:Dudzik} alerts us to one obvious issue. At least in the setting of $K$-affinoid algebras, the distinction between Berkovich spaces vs. adic spaces is analogous to the distinction between maximal filters vs. prime filters. Recalling Discussion~\ref{dis:typesMODELSfilters}, this raises questions about whether the model theorist's language of definable types is suited for analysing adic spaces since types properly correspond to ultrafilters and not prime filters. Given what we know about localic spaces and geometric logic, Theorem~\ref{thm:Dudzik} leads us to wonder if the point-free approach would be better suited (although some care should be taken regarding the distinction between prime vs. completely prime filters). 
	
%%% Hrushovski: Berkovich points strictly translate to stably dominatred types but with adic spaces over arbitrary bases there is no such 1-1 connection. Of course one can still study the definable ones.  
%%% Thought: there are two levels of unique extensions going on. How can we complete an ultrafilter to another one (via definability of types), but also how can we complete a prime filter to an ultrafilter (which is something made precise in the language of overconvergent filters).

\end{discussion}

%\newpage 
\appendix

\section{Constructing multiplicative seminorms from $K$-seminorms}\label{Appendix:BerkDisc}

We now provide the rest of the details of the proof of Theorem~\ref{thm:BerkovichDISC}. For clarity, we review the key steps of the argument. Recall the main goal: we wish to show that $\M(\A)$ is classically equivalent to $\overleftarrow{\M(\Alin)}$. After declaring our reliance on Assumption $(\star)$, the argument proceeds by construction: given $|\cdot|_x\in\M(\A)$, one defines $|\cdot|_x\big|_{\Alin}$ on $\Alin$ by taking the obvious restriction, whereas given $|\cdot|_\F\in\overleftarrow{\M(\Alin)}$, we start by first extending $|\cdot|_\F$ to a bounded multiplicative seminorm on $K[T]$, before finally defining $\widetilde{|\cdot|_\F}$ on $\A$. It is fairly obvious that if both constructions are well-defined, then they are inverse to each other because they are both determined by their values on the linear polynomials (which would complete the proof). In addition, it is clear that the restriction $|\cdot|_x\big|_{\Alin}$ defines a $K$-seminorm, but more effort is needed to check that $\widetilde{|\cdot|_\F}$ satisfies the required properties.\newline 

We organise the argument into the following two claims. 

\begin{claim}\label{claim:K[T]seminorm} The extension map $|\cdot|_\F\colon K[T]\to [0,\infty)$  (given by Equation~\eqref{eq:polynomialEXT}) defines a bounded multiplicative seminorm on $K[T]$, satisfying the ultrametric inequality. 
\end{claim}
\begin{proof}[Proof of Claim] Our argument relies on the explicit characterisation of $|\cdot|_\F$ to perform the required checks. 
	\subsubsection*{Step 1: Working ``level-wise''} Fix a ball $B_{q}(k)\in\F$, and define the obvious extension of $|\cdot|_{B_{q}(k)}$ from $\Alin$ to $K[T]$:
	\begin{align}
	|\cdot|_{B_{q}(k)}\colon K[T]&\longrightarrow [0,\infty)\\
	f &\longmapsto |c|\cdot \prod^m_{j=1} |T-b_j|_{B_{q}(k)} \nonumber 
	\end{align}
	Following the cue of Remark~\ref{rem:AVOIDsups}, we may also express $f$ as a finite power series centred at $k$:
	\begin{equation}\label{eq:repFINITEPOWERSERIES}
	f=\sum_{i=0}^{m}c_i(T-k)^i,
	\end{equation}
	and define a new map
	\begin{align}\label{eq:classicalDEFnorm}
	\widehat{|\cdot|_{B_{q}(k)}}\colon K[T]&\longrightarrow [0,\infty)\\
	f &\longmapsto  \max_i|c_i|q^i \nonumber 
	\end{align}
	We claim that $\widehat{|\cdot|_{B_q(k)}}$ defines a multiplicative seminorm (though not necessarily bounded, since $q$ is an arbitrary positive rational). This follows from noting: 
	\begin{itemize}
		\item \emph{$\widehat{|0|_{B_q(k)}}=0$ and $\widehat{|1|_{B_q(k)}}=1$}. Obvious.
		\item \emph{$\widehat{|\cdot|_{B_q(k)}}$ satisfies the ultrametric inequality.} Straightforward, but we elaborate for completeness. Given $f,f'\in K[T]$, assume WLOG that $\deg(f)=m\geq m'=\deg(f')$. Then, add their corresponding finite power series (as in Equation~\eqref{eq:repFINITEPOWERSERIES}) and
		compute:
		\begin{align*}
		\widehat{|f+f'|_{B_{q}(k)}}&= \widehat{|\sum_{i=0}^{m}c_i(T-k)^i + \sum_{i=0}^{m'}c'_i(T-k)^i|_{B_q(k)}} \\
		&= \widehat{|\sum_{i=0}^{m}(c_i+c'_i)(T-k)^i|_{B_q(k)}} \qquad\qquad\qquad \qquad\quad\,\,\,\,\,\,\, \text{[writing $c'_i=0$ for all $i>m'$]}\\
		& =  \max_{i}|c_i+c'_i|q^i \qquad\qquad\qquad\qquad\qquad\qquad  \qquad\quad\,\,\text{[by Definition of $\widehat{|\cdot|_{B_q(k)}}$]}\\
		&\leq \max_i\left\{\max\{|c_i|\,,\,|c'_i|\}\cdot q^i\right\} \qquad\qquad\qquad\qquad\qquad\, \, \text{[since $|c_i+c'_i|\leq\max\{|c_i|,|c_i'|\}$]}\\
		&= \max \{\widehat{|f|_{B_q(k)}}\,,\, \widehat{|f'|_{B_q(k)}}\} \qquad\qquad\qquad\qquad\qquad\,\quad  \text{[since the $\max$'s commute]}.
		\end{align*}
		\item \emph{$\widehat{|\cdot|_{B_q(k)}}$ is multiplicative.} 
		Since $f,f'$ are both polynomials, there exists $v,w\in\N$ such that
		
		$$\widehat{|f|_{B_q(k)}}= \max_i |c_i|q^i= |c_v|q^v$$
		
		$$\widehat{|f'|_{B_q(k)}}=\max_i|c'_i|q^i=|d_w|q^w$$
Pick $v,w$ to be first maximising indices in the lexicographic order.\footnote{To illustrate: if $\widehat{|f|_{B_{q}(k)}}=|c_i|q^i=|c_{i+1}|q^{i+1}$ and $|c_{j}|q^j< |c_i|q^i$ for all $j<i$, then we set $v:=i$ instead of $i+1$.} For explicitness, we write:
			\begin{equation}
			\widehat{|f|_{B_q(k)}}\cdot \widehat{|f'|_{B_q(k)}}=|c_v||d_w|q^{v+w}.
			\end{equation}
	\begin{equation}
	\widehat{|f\cdot f'|_{B_q(k)}}=\widehat{|\left(\sum c_id_{n-i}\right)(T-k)^n|_{B_q(k)}}=\max_{n} \big|\sum c_id_{n-i}\big|q^n %= |\sum c_id_{u-i}|q^u,
	\end{equation}
%which clearly attains its maximum at $u=v+w$.
The ultrametric inequality on $K$ almost immediately gives $\widehat{|f\cdot f'|_{B_q(k)}}\leq \widehat{|f|_{B_q(k)}}\cdot \widehat{|f'|_{B_q(k)}}$. For the reverse inequality, denote $u:=v+w$. Note that $\sum_{i}c_id_{u-i}$ equals $c_vd_w $ plus other terms of the form $c_{i'}d_{u-i'}$, and so one of these two cases must occur:
\begin{itemize}
	\item[] \textbf{Case 1: }\emph{$i'$ precedes $v$ in the lexicographic order.} In which case $|d_{u-i'}|q^{u-i'}\leq |d_v|q^v$ and $|c_{i'}|q^{i'}<|c_v|q^v$;
	\item[] \textbf{Case 2: }\emph{$u-i'$ precedes $w$ in the lexicographic order.} In which case $|d_{u-i'}|q^{u-i'}< |d_v|q^v$ and
	$|c_{i'}|q^{i'}\leq |c_v|q^v$.
	\end{itemize}
Either way, the ultrametric inequality once again gives $|\sum c_id_{u-i}-c_vd_w|<|c_vd_w|$, and so $|\sum c_id_{u-i}|=|c_vd_w|$.\footnote{Why? This follows from the fact that $|x+y|=\max\{x,y\}$ if $|x|<|y|$ if $|\cdot|$ is ultrametric. To prove the claim, notice $|y|\leq \max\{|x+y|,|-x|\}$, and so $|x|<|y|$ implies $|y|\leq |x+y|$. But since $y\leq |x+y|\leq \max\{|x|,|y|\}$, conclude $|y|=|x+y|$.} Hence, conclude that $\widehat{|f|_{B_q(k)}}\cdot \widehat{|f'|_{B_q(k)}}\leq \widehat{|f\cdot f'|_{B_q(k)}}$, verifying multiplicativity of $\widehat{|\cdot|_{B_q(k)}}$. 

	\end{itemize}
	
	Finally, we claim that:
	\begin{equation}\label{eq:ballLEVELid}
	|\cdot|_{B_q(k)}=\widehat{|\cdot|_{B_q(k)}}.
	\end{equation}
	Why? Since $|\cdot|_{B_{q}(k)}$ is multiplicative and $K$ is algebraically closed, it suffices to show they agree on the level of linear polynomials. But this is clear since
	\begin{equation}
	|T-a|_{B_q(k)}=\max\{|k-a|,q\}=\widehat{|T-a|_{B_q(k)}}, \qquad\quad \text{for any $T-a$}.
	\end{equation}
Hence, conclude that $|\cdot|_{B_q(k)}$ is in fact a multiplicative seminorm satisfying the ultrametric inequality. 
	
	\subsubsection*{Step 2: Lifting to $\F$} The argument is similar to the proof of Claim~\ref{claim:GETk-SEMINORM}. To show that $|\cdot|_\F$ is a muliplicative seminorm on $K[T]$ satisfying the ultrametric inequality, this amounts to checking a list of properties. But by Step 1, we know that these properties already hold for $|\cdot|_{B_q(k)}$, for all $B_{q}(k)\in \F$. Hence, since $|\cdot|_\F=\inf_{B_q(k)\in\F}|\cdot|_{B_q(k)}$, observe that these properties are respected by taking $\inf$'s, and conclude that they hold for $|\cdot|_\F$ as well. 
	
	\begin{comment}
	
	To illustrate, suppose $f,f'\in K[T]$. By Step 1, we know that the ultrametric inequality holds for \emph{all} $|\cdot|_{B_q(k)}$ where $B_{q}(k)\in\F$:
	\begin{equation}
	|f+f'|_{B_{q}(k)}\leq \max\{|f|_{B_{q}(k)}, |f'|_{B_{q}(k)}\}
	\end{equation}
	Then, since taking $\inf$'s respects weak inequalities, conclude that
	\begin{equation}
	|f+f'|_{\F}\leq \max\{|f|_{\F}, |f'|_{\F}\}.
	\end{equation}
\end{comment}

	\subsubsection*{Step 3: $|\cdot|_\F$ is bounded} We do not get boundedness of $|\cdot|_\F$ from Step 1, so this must be checked separately. But since $\F$ is $R$-good, this implies 
	\begin{equation}
	|T-a|_\F\leq \max\{|a|,R\} = ||T-a||,
	\end{equation}
	where $||\cdot||$ is the Gauss norm restricted to $K[T]$.\footnote{Why? See proof of Claim~\ref{claim:GETk-SEMINORM}.} Since all polynomials factor into linear polynomials, and since both $||\cdot||$ and $|\cdot|_\F$ are multiplicative seminorms (by Step 2),  deduce that
	\begin{equation}
	|f|_\F\leq ||f||, \qquad\qquad  f\in K[T].
	\end{equation}
	This finishes the proof of our claim.
\end{proof}

\begin{claim}\label{claim:POWERseminorm} The construction $\widetilde{|\cdot|_\F}$ defines a bounded multiplicative seminorm on $\A$.
\end{claim}
\begin{proof}[Proof of Claim] The fact that $\widetilde{|0|_\F}=0$ and $\widetilde{|1|_\F}=1$ is obvious by construction. As for the other properties:
	\begin{itemize}
		\item \emph{$\widetilde{|\cdot|_\F}$ is a well-defined map.} Since $\widetilde{|\cdot|_\F}$ takes values in $[0,\infty)$, we need to show that the limit of Equation~\eqref{eq:powerseriesEXT} exists for $f\in\A$. For explicitness, let $f=\sum^\infty_{i=0}a_iT^i$. By Claim~\ref{claim:K[T]seminorm}, we know that $|\cdot|_\F$ is bounded, and also satisfies the ultrametric inequality on $K[T]$. Hence, for any natural numbers $M< N$, a telescoping series argument yields 
\begin{comment}
\footnote{How? First observe
			$$\big| |\sum^{M+1}_{i=0}a_iT^i|_\F - |\sum^M_{i=0}a_iT^i|_\F\big|\leq \big| |a_{M+1} T^{M+1}|_\F + |\sum^M_{i=0}a_iT^i|_\F - |\sum^M_{i=0}a_iT^i|_\F \big|=\big| |a_{M+1} T^{M+1}|_\F\big|$$  
			by the triangle inequality. Denote $s_N=|\sum^N_{i=0}a_iT^i|_\F$. A telescoping series argument then gives
				$$\big| s_{N} - s_{M}\big|= \big| s_N - s_{N-1} + s_{N_1} + \dots - s_{M+1} + s_{M+1} - s_{M}
				\big|\leq \big| \sum_{i=M+1}^{N} |a_{M+1} T^{M+1}|_\F\big|\leq \max_{M+1\leq i\leq N}\{|a_i|R^i\}.$$  } 
				\end{comment}	
		\begin{equation}
		\big| |\sum^N_{i=0}a_iT^i|_\F - |\sum^M_{i=0}a_iT^i|_\F\big| \leq \max_{M+1\leq i\leq N}\{|a_i|R^i\}.
		\end{equation}
		Since $f\in\A$, we know $|a_i|R^i\to 0$ by definition. Hence, $\{|\sum^n_{i=0}a_iT^i|_\F\}_{n\in\N}$ is a Cauchy sequence and thus converges to a limit in $[0,\infty)$.
		\item \emph{Bounded.}  Let $f\in\A$ where $f=\sum^\infty_{i=0}a_iT^i$. Since
		\begin{equation}
		\lim_{n\to\infty} ||\sum_{i=0}^{n}a_iT_i||=||f||,
		\end{equation}
		and since 
		\begin{equation}
		|\sum_{i=0}^{n}a_iT_i|_\F\leq ||\sum_{i=0}^{n}a_iT_i||, \qquad \text{for all $n$,}
		\end{equation}
		by Claim~\ref{claim:K[T]seminorm}, conclude that 
		$\widetilde{|\cdot|_\F}\leq ||\cdot ||$. 
		\item \emph{Ultrametric Inequality.} This also follows from $|\cdot|_\F$ satisfying the ultrametric inequality. Indeed, given $f,f'\in\A$, compute:
		\begin{align*} \widetilde{|f+f'|}_\F & = \lim_{n\to\infty} |\sum^n_{i=0}a_iT^i + \sum^n_{i=0}b_iT^i|_\F \\
		&\leq \lim_{n\to\infty} \max\{|\sum^n_{i=0}a_iT^i|_\F, |\sum^n_{i=0}b_iT^i|_\F\} \\
		&= \max\{\widetilde{|f|}_\F,\widetilde{|f'|}_\F\}
		\end{align*}
		with representations $f=\sum^\infty_{i=0}a_iT^i$ and $f'=\sum^\infty_{i=0}b_iT^i$.
		
		\item \emph{Multiplicativity.} For orientation, note $f\in\A$ converges absolutely with respect to $||\cdot||$ since 
		\begin{equation}
		||f||=\norm{\sum^{\infty}_{i=0}a_i T^i}= \max_{i}|a_i|R^i = \norm{\sum^{\infty}_{i=0}|a_i|T^i}
		\end{equation}
		and $|a_i|R^i\to 0$. Extending this observation, since $\widetilde{|\cdot|_\F}\leq ||\cdot||$ on $\A$, one easily adapts the proof of Merten's Theorem for Cauchy products (for infinite series) to prove multiplicativity of $\widetilde{|\cdot|_\F}$.

	\end{itemize}
	This completes the proof of the claim.
\end{proof}

\begin{remark} Readers familiar with Berkovich geometry may recognise parallels between our proof of Claim~\ref{claim:POWERseminorm} and the standard proof of the homeomorphism
	\[\AffBerk^1\cong\bigcup_{R>0}\M(K\{R^{-1}T\}).\]
\end{remark}

%\newpage 

\newpage 
\bibliography{thesis}

\end{document}